\documentclass[11pt,a4paper,
fleqn,reqno]{amsart}                 
\usepackage[T1]{fontenc}                   
\usepackage[utf8]{inputenc}                 
\usepackage[english]{babel}    
 \usepackage{csquotes}
\usepackage{graphicx}                      %
\usepackage[font=small]{quoting}    
\usepackage{float}
\usepackage{caption}  
\usepackage[top=1.2in,bottom=1.8in,left=1.2in,right=1.2in]{geometry}                  
\usepackage{verbatim}

\usepackage{color}
\usepackage{xcolor}
\usepackage{picture}
\usepackage{amsmath,amsthm,amsfonts,amssymb}
\usepackage{comment}
\usepackage[draft]{hyperref}
\hypersetup{colorlinks,linkcolor={magenta}, citecolor={red},urlcolor={black}}  
\usepackage{enumitem}
\usepackage{mathtools}
\usepackage[url=false,doi=false,isbn=false, giveninits=true, 
style=numeric, 
maxbibnames=99]{biblatex}
\newtheorem{thm}{Theorem}[section] 
\newtheorem{lem}[thm]{Lemma} 
 
\newtheorem{prop}[thm]{Proposition} 
 
\newtheorem{defn}{Definition}[section]
\theoremstyle{definition} 
 
\theoremstyle{remark}

\theoremstyle{definition}

\def\O{\Omega}
\def\S{\Sigma} 
\def\n{\nabla}

\def\p{\partial}

\def\a{\alpha}

\def\n{\nabla}

\def\O{\Omega}
\def\p{\partial}
\def\a{\alpha}

\def\d{\delta}

\def\l{\lambda}
\def\s{\sigma}

\def\ov{\overline}

\def\n{\nabla}
\def\<{\langle}
\def\>{\rangle}

\def\n{\nabla}

\def\RR{\mathbb{R}}

\def\SS{\mathbb{S}}

\def\O{\Omega}
\def\p{\partial}

\def\a{\alpha}

\def\d{\delta}
\def\l{\lambda}
\def\K{\mathcal{K}}
\def\s{\sigma}

\def\ov{\overline}

\def\wh{\widehat}

\def\C{\mathcal{C}}

\def\ol{\overline}
\def\wt{\widetilde}
\def\K{\mathcal{K}}

{\left\{\begin{array}{@{}l@{}}}{\end{array}\right.}
\patchcmd{\abstract}{\scshape\abstractname}{\textbf{\abstractname}}{}{}
\makeatletter 
\def\@makefnmark{} 
\makeatother

\numberwithin{equation}{section}
\numberwithin{exa}{section}

\makeatletter
\@namedef{subjclassname@2020}{\textup{2020} Mathematics Subject Classification}
\makeatother

\begin{document}
	\title[The capillary $L_p$-Minkowski problem]
	{The capillary $L_p$-Minkowski problem}

\author[X. Mei]{Xinqun Mei}\address[X. Mei]{Key Laboratory of Pure and Applied Mathematics, School of Mathematical Sciences, Peking University,  Beijing, 100871, P.R. China}\email{\href{qunmath@pku.edu.cn}{qunmath@pku.edu.cn}}
	
\author[G. Wang]{Guofang Wang}\address[G. Wang]{Mathematisches Institut, Albert-Ludwigs-Universität Freiburg, Freiburg im Breisgau, 79104, Germany}\email{\href{guofang.wang@math.uni-freiburg.de}{guofang.wang@math.uni-freiburg.de}}

\author[L. Weng]{Liangjun Weng}\address[L. Weng]{Centro di Ricerca Matematica Ennio De Giorgi, Scuola Normale Superiore, Pisa, 56126, Italy  \& Dipartimento di Matematica, Universit\`a di Pisa, Pisa, 56127, Italy}\email{\href{mailto:liangjun.weng@sns.it}  {liangjun.weng@sns.it}}
	


\begin{abstract}
This paper is a continuation of our recent work [Adv. Math. 469 (2025), Paper No. 110230] concerning the capillary Minkowski problem. We propose, in this paper, a capillary $L_p$-Minkowski problem for $p\in \mathbb{R}$, which seeks to find a capillary convex body with a prescribed capillary $L_p$-surface area measure in the Euclidean half-space. This formulation provides a natural Robin boundary analogue of the classical $L_p$-Minkowski problem introduced by Lutwak [J. Differential Geom. 38 (1993), no. 1, 131--150]. For $p>1$, we resolve the capillary $L_p$-Minkowski problem in the smooth category by reducing it to a Monge--Amp\`ere equation with a Robin boundary condition on the unit spherical cap.
\end{abstract}

\maketitle


\section{Introduction}
 
The classical Minkowski problem aims to determine a convex body whose surface area measure corresponds to a given spherical Borel measure. This problem has been completely resolved under a necessary and sufficient condition, through the significant contributions of Minkowski \cite{Min}, Alexandrov \cite{Alex}, Lewy \cite{Lewy}, Nirenberg \cite{Nire}, Pogorelov \cite{Pog52}, Cheng--Yau \cite{CY76}, and many other mathematicians. The study of the Minkowski problem has been a key inspiration in the development of Monge--Amp\`ere equation and the Brunn--Minkowski theory. In 1962, Firey \cite{Firey} introduced the concept of $p$-sums of convex bodies, and Lutwak \cite{Lut93} subsequently demonstrated that a corresponding $L_p$ Brunn--Minkowski theory exists for each of Firey's $p$-sums and $L_p$-surface area measure for $p\geq 1$. Namely, for a convex body $K\subset \RR^{n+1}$ containing the origin as its interior, the $L_p$-surface area measure $d S_{p}$ on $\SS^n$ is defined as
\begin{eqnarray}\label{Lp-surface-area-mea}
    dS_{p}=h^{1-p} dS,
\end{eqnarray}where $h$ is the support function of $K$, and $d S$ is the standard surface area measure of $\p K$. This naturally leads to the formulation of the $L_p$-Minkowski  problem:

\

{\it\noindent {\bf $L_p$-Minkowski problem}.
Let $\mathcal{K}_{0}$ denote the set of convex bodies in $\RR^{n+1}$ containing the origin in their interior.  Given a finite non-trivial Borel measure $\vartheta$ on $\SS^{n}$,  does there exist a convex body $K\in \mathcal{K}_{0}$, such that its $L_p$-surface area measure $d S_{p}$ coincides with the given measure $\vartheta$? }

\

This problem has become one of the main research fields in modern convex body theory and geometric partial differential equations. In particular, $p=1$ corresponds to the classical Minkowski problem concerning the surface area measure, while $p=0$ corresponds to the logarithmic Minkowski problem concerning the cone volume measure. If we further assume that $\vartheta=f(x) d\sigma $, where $d\sigma$ is the standard spherical area measure  and $f(x)$ is a non-negative 
 function on $\SS^{n}$, then in the smooth category, the $L_p$-Minkowski problem can be transformed into solving the following Monge--Amp\`ere  equation
\begin{eqnarray}\label{det}
    \det ( \n^{2}h+h\sigma)=fh^{p-1},\quad \text{on}~\SS^{n},
\end{eqnarray}
where $ \n$ is the covariant derivative with respect to the spherical metric $\sigma$  on $\SS^{n}$. 

The $L_p$-Minkowski problem and its various variants have been considered important problems in convex geometry, and many significant results have been achieved in the past decades. Given the vast scope of developments, it is nearly impossible to provide a comprehensive account here, so we just outline some progress closely related to this problem and recommend the interested reader to the comprehensive book by Schneider \cite{Sch} and a more recent one by Böröczky--Figalli--Ramos \cite{BRF}. When $p>1$, the $L_p$-Minkowski problem was solved by Lutwak \cite{Lut93} for symmetric measures,  Chou--Wang \cite{CW2006-AIM}, and Hug--Lutwak--Yang--Zhang \cite{HLYZ}  for general measures.
In the smooth category, the regularity of the solutions was established by Lutwak--Oliker \cite{LO95} for symmetric measures. Later, Guan--Lin \cite{GL2000} and Chou--Wang \cite{CW2006-AIM} removed the symmetric assumption in \cite{LO95} if $p\geq n+1$. When $p=0$, the $L_p$-Minkowski problem is the logarithmic Minkowski problem, which is equivalent to finding a convex body with a given cone-volume measure. In \cite{BLYZ}, Böröczky--Lutwak--Yang--Zhang obtained necessary and sufficient conditions for the existence of solutions when the prescribed measure is a symmetric Borel measure. See also a recent related work in \cite{Böröczky_Guan_2023} for an alternative anisotropic Gauss curvature flow approach.  For more related results on the logarithmic Minkowski problem,  see, for instance \cite{BHZ, CFL, Chen-Li-Zhu2019TAMS, Sta-1, Sta-2, Zhu} and references therein. When $p=-n-1$, this problem is known as the centro-affine Minkowski problem, which is related to the $SL(n)$ invariant centro-affine curvature whose reciprocal is $h^{n+2}\det(\n^{2}h+h\sigma)$, see for instance \cite{Hug}  and \cite{Ludwig}. For more existence results concerning the centro-affine Minkowski problem, see also  \cite{GLW, JLZ} and references therein. 
In \cite{CW2006-AIM},  Chou--Wang found a Kazdan--Warner type condition which implies that for a general positive function $f$, the centro-affine Minkowski problem may admit no solutions. 
Du \cite{Du} constructed an example with some positive function $f\in C^{\alpha}(\mathbb{S}^{n})$, such that the planar centro-affine Minkowski problem is not solvable.  Until now, the existence and the multiplicity of solutions to the centro-affine Minkowski problem remain a challenging problem. When $p<-n-1$, a recent breakthrough by Guang--Li--Wang \cite{GLW-1}  showed  that for any positive $C^{2}$ function $f$, there exists a $C^{4}$ solution to Eq. \eqref{det}.
There have also been numerous developments concerning the $L_p$-Minkowski problem over the past decades; see, for instance, \cite{ACW, Ben98, BBC, BBCY, Boro-Saroglou2024, Bryan-Ivaki-Scheuer2019, CLZ, HLW, Huang-Lu2013, Hug-LYZ, Jian-Lu-Wang2015-AIM, Kolesnikov-milman2022, Lu, LH,  LW, Lutwak-Yang-Zhang2004, MWW-Quotient, Zhu-2}, and the references therein.

\subsection{Setup and the problem} Following the notations as \cite{MWW-AIM}, we now describe our setting and problem in more detail. 
Let $\RR^{n+1}_+=\{x\in \RR^{n+1}\,|\, x_{n+1}>0\}$ be the upper Euclidean half-space, and $\S\subset \overline{\RR^{n+1}_+}$  a properly embedded, smooth compact hypersurface with boundary such that 
$$\text{int}(\S)\subset \RR^{n+1}_{+}~~~~~\text{and}~~~~~\p \Sigma\subset \p {\RR^{n+1}_{+}}.$$
 We call $\S\subset \ol{\RR^{n+1}_+}$  a \textit{capillary hypersurface} if $\S$ intersects $\p\RR^{n+1}_+$ at a constant contact angle $\theta\in (0, \pi)$. A simple example of a capillary hypersurface is the spherical cap $\C_{\theta}$:  
      \begin{eqnarray*}
		\C_{\theta }  \coloneqq \left\{ \xi\in \ov{\mathbb{R}^{n+1}_+} ~\mid~ |\xi-\cos\theta e|= 1 \right\},
	\end{eqnarray*}
where $e\coloneqq -E_{n+1}=(0, \cdots, 0,-1)$, $E_{n+1}$ is the $(n+1)$-th unit coordinate vector in $\ov{\RR^{n+1}_+}$. We call $\widehat{\S}$ a \textit{capillary convex body} if $\widehat{\S}$ is a bounded closed region in $\ol{\RR^{n+1}_+}$ enclosed by the convex capillary hypersurface $\S$ and $\p\RR^{n+1}_+$, and denote $\wh{\p\S}\coloneqq \p (\wh\S)\setminus \S\subset \p\RR^{n+1}_+$. As in \cite[Section 2]{MWW-AIM}, we use the \textit{capillary Gauss map} 
      \begin{eqnarray*}
          \tilde \nu: &\S &\to \C_\theta  \\ & X &\mapsto \nu(X)+\cos\theta e, 
      \end{eqnarray*} where $\nu$ is the usual Gauss map. It turns out that $\tilde \nu$ is a diffeomorphism map (see \cite[Section 2]{MWWX}), thus we can reparametrize $\S$ using its inverse on $\C_\theta$. Hence, we also view the usual support function $h$ of $\S$ as a function defined on $\C_\theta$.  Denote $\mathcal{K}_\theta$ as the set of all capillary convex bodies in $\ol{\mathbb{R}^{n+1}_{+}}$, and $\mathcal{K}^{\circ}_{\theta}$ as the family of capillary convex bodies with the origin as an interior point of the flat part of their boundary.  
      
      The primary goal of this paper is to study the \textit{capillary $L_p$-surface area measure} $dS^c_{p}$ for a convex capillary body in $\ol{\RR^{n+1}_+}$, which is defined  by \begin{eqnarray}\label{eqn-Lp-surface-area-measure}
          dS_{p}^c\coloneqq  \ell^p  h^{1-p}\det(\n^{2}h+h\sigma)d\sigma,
      \end{eqnarray} 
where $d\sigma$ is the standard area element on $\C_\theta$, $$\ell  \coloneqq \sin^{2}\theta+\cos\theta\<\xi, e\>,$$ and $h$ are the support functions of $\C_\theta$ and $\S$ respectively. See Eq. \eqref{capillary-L-p-area} and Section \ref{sec2} for more details.  
 Thus it is natural to study the regular $L_p$-Minkowski problem for $d S_{p}^c$, which serves as a capillary counterpart of the classical $L_p$-Minkowski problem. 
      
\
	 
\noindent	\textbf{Capillary $L_p$-Minkowski problem.} \textit{  Given a positive smooth function $f$ on $\C_{\theta}$,  does  there exist a capillary convex body $\wh\S\in\mathcal{K}_{\theta}^{\circ}$ such that its capillary $L_p$-surface area measure $dS^c_{p}$ equals to $f\ell^p d\sigma$ ? }

\ 

This problem not only provides a natural boundary analogue of Lutwak's $L_p$-theory but also poses novel analytical challenges, notably the resolution of a Monge--Amp\`ere equation under a Robin-type boundary condition. By applying a similar argument to that in \cite[Proposition~2.4]{MWW-AIM} and \cite[Lemma~2.4, Proposition~2.6]{MWWX},  we reduce the capillary $L_p$-Minkowski problem to a problem of solving the existence of convex solutions to a fully nonlinear elliptic PDE with a Robin boundary condition. More precisely, it corresponds to the solvability of the following Monge--Amp\`ere equation with a Robin boundary value condition:  
\begin{eqnarray}\label{eq-monge-ampere-Lp}
\left\{
		\begin{array}{rcll}\vspace{2mm}\displaystyle
			\det(\n^2 h+h\s )&=&  f h^{p-1},& \quad    \hbox{ in } \C_{\theta},\\
			\n_\mu h &=& \cot\theta  h, &\quad  \hbox{ on } \p \C_\theta,
		\end{array}\right.
	\end{eqnarray}  
where 
$\n h$ and $\n^2h$ are the gradient and the Hessian of $h$ on $\C_\theta$ with respect to the standard spherical metric $\s$ on $\C_{\theta}$ respectively,   and $\mu$ is the unit outward normal of $\p \C_\theta \subset \C_\theta$. In this paper, by a convex solution to Eq. \eqref{eq-monge-ampere-Lp}, we mean that
\begin{eqnarray*} 
   A\coloneqq\n^2 h+h\s >0,~~~ ~{\text{in}}
   ~\C_{\theta}.
\end{eqnarray*}	When $\theta =\frac\pi 2$, this problem can be reduced to the classical closed case by a reflection argument, see Guan--Lin \cite{GL2000}, and Chou--Wang \cite{CW2006-AIM} among many others. Hence we focus on the general case $\theta \neq \frac{\pi}{2}$.  
To present the main result concisely, we introduce a definition that is a capillary adaptation of the even function on $\SS^n$, considered already in \cite{mei-wang-weng-2025convexcapillaryhypersurfacesprescribed}.
\begin{defn}
    Let $f\in C^{2}(\C_{\theta})$ and for $\xi=(\xi_{1},\cdots, \xi_{n},\xi_{n+1})\in\C_{\theta}$, denote $\widehat{\xi}\coloneqq (-\xi_{1},\cdots, -\xi_{n}, \xi_{n+1})$. If $f$ satisfies $f(\xi)=f(\widehat{\xi})$, we call $f$  a capillary even function on $\C_{\theta}$.    A capillary convex body $\wh\S\in \K^\circ_\theta$ is called capillary symmetric if its associated support function is a capillary even function.
\end{defn}

\subsection{Main results} 
We solve the capillary $L_p$-Minkowski problem for $p>1$, as formulated in Eq.~\eqref{eq-monge-ampere-Lp}. The case $p=1$ was previously addressed in \cite{MWW-AIM}. This work thus extends the $L_p$-Minkowski theory to the capillary setting, uncovering new interactions between $L_p$-surface area measure and capillary geometry. The main theorem of this paper extends the result in \cite{MWW-AIM} to all $p>1$ and is stated below.
\begin{thm}\label{thm-1}
Let $p>1$ and $\theta\in (0, \frac{\pi}{2})$. For any positive smooth function $f$ defined on $\C_{\theta}$.
\begin{enumerate}
    \item  If $p>n+1$, then  there exists a unique  smooth solution $h$ solving Eq. \eqref{eq-monge-ampere-Lp}. Moreover, there exists a positive constant $C$, depending only on $n, p, \min\limits_{\C_{\theta}}f$ and $\|f\|_{C^{2}(\C_{\theta})}$, such that
    \begin{eqnarray*}
        \|h\|_{C^{2}(\C_{\theta})}\leq C.
    \end{eqnarray*}
\item If $p=n+1$, then there exists a unique (up to a dilation) positive smooth solution $h$  and a positive constant $\gamma$ solving  
\begin{eqnarray*}\left\{
		\begin{array}{rcll}\vspace{2mm}\displaystyle
			\det(\n^2 h+h\s )&=&  \gamma f h^{n} ,& \quad    \hbox{ in } \C_{\theta},\\
			\n_\mu h&=&\cot\theta h, &\quad  \hbox{ on } \p \C_\theta.
		\end{array}\right.
  \end{eqnarray*}
 \item If $1<p<n+1$,  assume that $f$ is capillary even, then there exists a smooth capillary even solution $h$ solving  Eq. \eqref{eq-monge-ampere-Lp}.
 \end{enumerate}
\end{thm}

As a direct consequence of Theorem~\ref{thm-1}, the capillary $L_p$-Minkowski problem is solved for all $p > 1$.

 \begin{thm}\label{thm-2}
     Let $p>1$ and $\theta\in(0, \frac{\pi}{2})$. For any positive smooth function $f$ on $\C_{\theta}$. 
     \begin{enumerate}
         \item When $p>n+1$, then there exists a unique capillary convex body such that its capillary $L_p$-surface area measure is equal to $ fd\sigma$.
         \item When $p=n+1$,  then there exists a unique positive constant $\gamma$ and a unique (up to a dilation) convex capillary body such that its capillary $L_p$-surface area measure is equal to $\gamma  fd\sigma$.
         \item When $1<p<n+1$, and $f$ is capillary even, then there exists a capillary symmetric convex body such that its capillary $L_p$-surface area measure is equal to $ fd\sigma$.
     \end{enumerate}
 \end{thm}

The restriction $\theta \in (0,\frac \pi 2)$ is crucially used in this paper to deal with the boundary condition, as in many related problems for capillary hypersurfaces, see for instance \cite{HWYZ, mei-wang-weng-2025convexcapillaryhypersurfacesprescribed, MWW-GCF, MWW-AIM, SW, WWX-MA2024, Weng-Xia2022}. The case $\theta \in (\frac \pi 2, \pi)$ remains as a challenging problem. From the analytical perspective, $\theta >\frac \pi 2$ leads to a sign inconsistency in the application of the maximum principle. Geometrically, in the case $\theta >\frac  \pi 2$,  convex capillary hypersurfaces fail to be convex in the sense of Busemann \cite{Busemann} for hypersurfaces with boundary.

\subsection{Outline of the proof}
We first point out that Theorem \ref{thm-2} follows directly from Theorem \ref{thm-1} and the identity \eqref{capillary-L-p-area} below. To prove Theorem \ref{thm-1}, we apply the continuity method, following the approach in \cite{GL2000, MWW-AIM}. The crucial ingredient is to establish a priori estimates for solutions to Eq. \eqref{eq-monge-ampere-Lp}. First of all, we establish uniform positive lower and upper bounds for solutions, as the absence of a uniform lower bound would cause Eq.~\eqref{eq-monge-ampere-Lp} to degenerate. When $p>n+1$, we proceed by using the capillary support function:
\begin{eqnarray*}
    u(\xi)\coloneqq \frac{h(\xi)}{\sin^{2}\theta+\cos\theta \<\xi, e\>},
\end{eqnarray*}
which satisfies a Monge--Amp\`ere type equation with homogeneous Neumann boundary conditions, see Eq. \eqref{eq-monge-ampere-Lp-u} below. This function is also crucially used in \cite{MWW-AIM}. By applying the maximum principle, we derive uniform lower and upper bounds for the solutions. When $1<p<n+1$, by combining the capillary isoperimetric inequality (cf. \cite[Corollary 1.3, and Eq. (1.7)]{MWW24-IMRN} or \cite[Chapter 19]{Maggi}),  H\" older's inequality and the special structure of Eq. \eqref{eq-monge-ampere-Lp}, we deduce a positive upper bound for the outer radius and a lower bound of the volume of the associated capillary convex body. Under the assumption of capillary evenness, we obtain positive lower and upper bounds for the solutions. Similar to the classical case, when $p\neq n+1$, the $C^1$ estimate follows directly from the convexity and $C^{0}$ estimate.  

When $p=n+1$, the dilation invariance of Eq. \eqref{eq-monge-ampere-Lp} implies the lack of an estimate for its solution. To address this problem,  we adopt the approach presented in  \cite[Section~3]{GL2000} or \cite[Section~4]{HMS2004}. The key step is to establish a logarithmic gradient estimate for the positive solution to Eq. \eqref{eq-monge-ampere-Lp} as
\begin{eqnarray}\label{log-h-est}
    \|\n \log h\|_{C^0(\C_\theta)}\leq C,
\end{eqnarray} where $C$ is independent of $p$.  In contrast to the classical closed case, the Robin boundary condition causes additional difficulties for us in obtaining this estimate. To address this, we introduce a new test function:  
\begin{eqnarray}\label{Phi}
  \Phi\coloneqq  \log\left|\n \left(\log h-\frac{\cot\theta}{2\theta} d^2_N\right) \right|^2+e^{\frac{s}{2\theta} d^2_N},
\end{eqnarray} where $d_N$ is the geodesic distance function on $\C_\theta$ and $s>0$ is a constant, see Section \ref{sec3.2} for more details. The function defined in \eqref{Phi} balances the interior and boundary terms, enabling effective use of the maximum principle. Thus, we are able to derive the desired estimate \eqref{log-h-est}.
For the $C^{2}$ estimate, we adopt the approach initiated by Lions-Trudinger-Urbas  \cite{LTU} for studying the Monge--Amp\`ere equation with a Neumann boundary condition on a uniformly convex domain in Euclidean space, see also \cite{MQ} for the analogous result for the $k$-Hessian equation with a Neumann boundary condition on a uniformly convex domain. The proof is primarily divided into two steps. First, we reduce the global  $C^{2}$ estimate to the boundary double normal estimate, as shown in Lemma \ref{lem c2 boundary}. Second, we deduce the boundary double normal $C^2$ estimate by constructing a suitable test function, as presented in Lemma \ref{c2 est}. We conclude by remarking that the range $\theta \in (0, \frac{\pi}{2})$ plays a crucial role in deriving the global $C^2$ estimate, as it guarantees that $\partial\C_{\theta} \subset \C_\theta$ is strictly convex. We point out that this range restriction is not required for us in establishing the $C^0$ and $C^1$ estimates for Eq. \eqref{eq-monge-ampere-Lp}, which remain valid for all $\theta\in(0,\pi)$.

\

Throughout the remainder of this paper, we assume that $p > 1$ and that $f: \C_{\theta} \to \RR$ is a positive smooth function. We shall indicate explicitly when the capillary evenness condition on $f$ is required.

\

\subsection*{Organization of the paper.} 
In Section \ref{sec2}, we introduce the capillary $L_p$-surface area measure of the capillary convex body and the corresponding capillary   $L_p$ Brunn--Minkowski inequality. In Section \ref{sec-3}, after transforming the capillary $L_p$-Minkowski problem into solving Eq. \eqref{eq-monge-ampere-Lp}, we will establish a priori estimates for the solution to Eq. \eqref{eq-monge-ampere-Lp}.
In Section \ref{sec-4}, we employ the continuity method to complete the proof of Theorem \ref{thm-1}.

\section{Capillary \texorpdfstring{$L_p$}{}-surface area measure}\label{sec2}  

In this section, we briefly review some properties of the capillary convex body and introduce the capillary $L_p$-surface area measure for the capillary convex body in $\ol{\RR^{n+1}_+}$. For more details concerning the geometry of capillary convex bodies, we refer the reader to \cite[Section 2]{MWW-AIM} and \cite[Section 2]{MWWX}. 
Let  $\widehat{\S}\in \K_{\theta}$ and $\nu$ be the unit outward normal of $\S$ in $\ol{\RR^{n+1}_+}$, the contact angle $\theta $ is defined by $$ \cos (\pi -\theta)=\langle \nu, e \rangle,\quad \text{along }~\partial\S,$$
	where $e= -E_{n+1}$ is the unit outward normal of $\p \ov{\RR^{n+1}_+}$.

If $\Sigma$ is strictly convex, the support function of $\Sigma$ is given by 
	\begin{eqnarray*}
		h(X)=\<X,\nu(X)\>.
	\end{eqnarray*}
	By the parametrization of $\tilde\nu^{-1}$, we can view $h$ as a function on $\C_\theta$ by
	\begin{eqnarray}\label{support}
		h(\xi) =\<X(\xi),\nu(X(\xi ))\>=
		\< \tilde \nu ^{-1}(\xi)  , \xi -\cos\theta e\>,
	\end{eqnarray}and satisfies
 \begin{eqnarray*} 
     \n_\mu h=\cot\theta h, ~~~~~~~~~\text{ on } \p \C_\theta.
 \end{eqnarray*} cf. \cite[Lemma 2.4]{MWWX}. We still call $h$ in \eqref{support} the support function of $\S$ (or of $\widehat{\S}$), which is only defined on $\C_\theta$. It is clear that for $\S= \C_\theta$, $\tilde \nu$ is the identity map and its support function is given by	\begin{eqnarray*}
		\ell(\xi)=\<\xi  , \xi -\cos\theta e\>=
		\sin ^2\theta + \cos\theta\< \xi, e\>.
	\end{eqnarray*}

Next, we recall Firey's $p$-sum in \cite{Firey} and adapt it to our capillary setting. 
\begin{defn}
    Given $\wh{\S}_1, \wh{\S}_2\in \K_{\theta}^\circ$, let $h_{\wh{\S}_1}(\cdot)$ and $h_{\wh{\S}_2}(\cdot)$ be the support functions of $\wh{\S}_1, \wh{\S}_2$ resp. Then  for $\lambda,\tau\geq 0$ with $\l^2+\tau^2\neq 0$ and $p\geq 1$, we define the  function 
    \begin{eqnarray*}
        h_{\lambda\cdot {\wh{\S}_1}+_{p}\tau\cdot {\wh{\S}_2}}(\cdot)\coloneqq  
 \left[\lambda h^{p}_{\wh{\S}_1}(\cdot)+\tau h^{p}_{\wh{\S}_2}(\cdot)\right]^{\frac 1 p}.
    \end{eqnarray*}
\end{defn}

By definition, it is obvious that the function $h_{\lambda\cdot {\wh{\S}_1}+_{p}\tau\cdot {\wh{\S}_2}}$ is positively homogeneous and nonnegative, moreover, it is the
 support function of some capillary convex body in $\K^\circ_\theta$.
\begin{prop}For $\wh\S_1,\wh\S_2\in \K^\circ_\theta$, then 
$h_{\lambda\cdot {\wh{\S}_1}+_{p}\tau\cdot {\wh{\S}_2}}$    is the support function of a capillary convex body in $\K_{\theta}^{\circ}$. As a result, we denote such a capillary convex body as $\lambda\cdot\widehat{\S}_{1}+_{p}\tau \cdot \widehat{\S}_{2}$.
\end{prop}
\begin{proof}
 From \cite[Definition 2.5]{MWWX}, one can easily see that $ h_{\lambda\cdot {\wh{\S}_1}+_{p}\tau\cdot {\wh{\S}_2}}(\cdot)$ is a capillary convex function. Namely, it is spherical convex in $\C_\theta$ if $p\geq 1$, and satisfies
 \begin{eqnarray*}
     \n_\mu h_{\lambda\cdot {\wh{\S}_1}+_{p}\tau\cdot {\wh{\S}_2}}=\cot\theta   h_{\lambda\cdot {\wh{\S}_1}+_{p}\tau\cdot {\wh{\S}_2}} ~~~ \text{ on }\p \C_\theta. 
 \end{eqnarray*}Hence by \cite[Proposition 2.6]{MWWX}, we know that $ h_{\lambda\cdot {\wh{\S}_1}+_{p}\tau\cdot {\wh{\S}_2}}$ is the support function of a capillary convex body in $\mathcal{K}^{\circ}_{\theta}$.
\end{proof}

 Recall we have the capillary area measure (the so-called wetting energy in \cite{Finn} or \cite[Eq. (1.2)]{MWW-AIM}) for the capillary convex body $\wh\S\in \K_{\theta}$ as
 \begin{eqnarray*}
    dS^c\coloneqq  dS^c_{\wh\S} \coloneqq \left(1+\cos\theta \<\nu(X), e\>\right) dS_{\wh\S},
 \end{eqnarray*}where $dS_{\wh\S}$ is the surface area measure of $\wh\S\cap \RR^{n+1}_+$ and the superscript $c$ indicates the presence of the \textit{capillary} effect in our notation. By means of the inverse of the capillary Gauss map $\tilde{\nu}$, and under the assumption that $\wh\S\cap \RR^{n+1}_+=\S$ is smooth and strictly convex, then it is also equivalent to the following form:
 \begin{eqnarray*}
 dS^c_{\wh\S} =  \ell \det(\n^2 h_{\wh\S}+h_{\wh\S} \s) d\s.
 \end{eqnarray*} 

 For $\varepsilon>0$, consider a smooth one-parameter family of capillary convex bodies: $K_\varepsilon\coloneqq K+_p\varepsilon L$, where $K,L\in \mathcal{K}^{\circ}_{\theta}$, then
 \begin{eqnarray*}
     \frac{d}{d\varepsilon}\bigg|_{\varepsilon=0} h_{K_\varepsilon}(\xi)= \frac 1 p h_K^{1-p}(\xi) h_L^p(\xi),~~~~\xi\in \C_\theta.
 \end{eqnarray*}Thus, the first variation of the volume under Firey's $L_p$ perturbation in direction $L$ yields
 \begin{eqnarray*}
      \frac{d}{d\varepsilon}\bigg|_{\varepsilon=0} \int_{\C_\theta} h_{K_{\varepsilon}}\det \left(\n^{2}(h_{K_{\varepsilon}})+h_{K_{\varepsilon}}\s \right) d \s&=& \frac {n+1} {p} \int_{\C_\theta}  h_L^p h_K^{1-p}\det\left(\n^{2}h_{K}+h_{K}\s\right) d\s.      \\&\eqqcolon & \frac {n+1} {p} \int_{\C_\theta}  h_L^p h_K^{1-p} dS_K . 
 \end{eqnarray*} Hence, this variational formula naturally motivates the introduction of a capillary $L_p$-surface area measure, which captures the weighted infinitesimal deformation of one capillary convex body $(L)$ by another $(K)$ under the influence of capillarity.

\begin{defn}[Capillary $L_p$-surface area measure]\label{def-Lp-surface-area}
    For $p\geq 1$, and $\widehat{\S}\in \mathcal{K}_{\theta}^{\circ}$, we define 
    \begin{eqnarray*}
    dS^c_{\wh\S,p}\coloneqq  h_{\wh\C_\theta}^p  h_{\wh\S}^{1-p} dS_{\wh\S} =\ell^p h_{\wh\S}^{1-p} dS_{\wh\S},
    \end{eqnarray*}and call it the capillary $L_p$-surface area measure of the capillary convex body $\wh\S$.
    \end{defn}
If without confusion, we will omit the subscript indicating its associated capillary convex body and write $h_{\wh\S}, dS_{\wh\S}, dS_{\wh\S,p}$ simply as $h, dS^c, dS^c_p$, respectively. 
    Through the inverse capillary Gauss map $\tilde \nu$, it is easy to see that 
    \begin{eqnarray}\label{capillary-L-p-area}
        dS^c_{p}=  \ell^p h ^{1-p}dS_{\wh\S}.
    \end{eqnarray}
In particular, when $\theta = \frac{\pi}{2}$, the measure $dS^{c}_{p}$ coincides with Lutwak’s classical $L_p$-surface area measure defined in \eqref{Lp-surface-area-mea}. Moreover, when $p=1$, we have
\begin{eqnarray*}
    dS^{c}_{1} = \ell\, dS_{\widehat{\Sigma}},
 \end{eqnarray*}
which agrees with the capillary area measure $m_\theta$ introduced in \cite[Page~5]{MWW-AIM}.

Now we define  \textit{capillary $L_p$-mixed volume} of $\wh\S_1,\wh\S_2\in \K^\circ_\theta$ by
$$V_p^c(\wh\S_1,\wh\S_2)\coloneqq\frac 1  {n+1} \int_{\C_\theta}  h_{\wh\S_2}^p h_{\wh\S_1}^{1-p} dS_{\wh\S_1}^c.  ~~~~~~~p\geq 1. $$ 
Analogously to the classical case, we have the capillary $L_p$ Brunn–Minkowski inequality, which may be of independent interest and potentially applicable to other problems for capillary hypersurfaces.
\begin{prop}
    Let $p>1$ and $\wh\S_1,\wh\S_2\in \K^\circ_\theta$, then the following hold:
\begin{enumerate}
\item  $ V_p^c(\wh\S_1,\wh\S_2)\geq V_p^c(\wh\S_1,\wh\S_1)^{1-p}~V_1^c(\wh\S_1,\wh\S_2)^p$.
 
\item $|\wh\S_1+_p\wh\S_2|\geq  
 \left(V_p^c(\wh\S_1,\wh\S_1)^p+V_1^c(\wh\S_1,\wh\S_2)^p \right)^{\frac 1 p}.$
\end{enumerate}
\end{prop}
The proofs are similar to those in \cite[Theorem 9.1.2, Theorem 9.1.3]{Sch} and are omitted here for conciseness. 

\ 

As a concluding remark to this section, we can introduce a more general area measure for the capillary convex body $\wh{\S} \in \K_\theta^\circ$, which we will refer to as the \textit{$k$-th capillary $L_p$ area measure}, denoted by $S^c_{k,p}(\wh\S,\cdot)$. In particular, when $k = n$, it coincides with the \textit{capillary $L_p$-surface area measure} defined in Definition~\ref{def-Lp-surface-area}, and when $p = 1$, it reduces to the $k$-th capillary area measure $dS_k^c$ studied in \cite{MWW-CM}. This naturally motivates the consideration of the capillary counterpart of the intermediate Christoffel–Minkowski problem associated with $p$-sums, namely, the \textit{capillary $L_p$ Christoffel–Minkowski problem}. To maintain our focus on the $L_p$-Minkowski problem and to avoid diverging from the main theme of this paper, we postpone the study of this problem to a forthcoming work \cite{MWW-Lp-CM}.

\ 

In the rest of this paper, for convenience, we adopt a local frame to express tensors and their covariant derivatives on the spherical cap $\C_\theta$. Throughout, indices appearing as subscripts on tensors indicate covariant differentiation. For instance, given an orthonormal frame $\{e_i\}_{i=1}^n$ on $\C_\theta$, the notation $h_{ij}$ stands for $\nabla^2 h(e_i,e_j)$, and $A_{ijk} \coloneqq \nabla_{e_k}A_{ij}$, and so forth. We employ the Einstein summation convention: repeated indices are implicitly summed over, regardless of whether they appear as upper or lower indices. In cases where ambiguity may arise, summation will be indicated explicitly.

\section{A priori estimates}\label{sec-3} \ 
In this section, we establish a priori estimates for any solution to Eq. \eqref{eq-monge-ampere-Lp}.  We state the main theorem as follows. 

\begin{thm} \label{thm priori est}
    Let $\theta\in (0, \frac{\pi}{2})$, and $p\neq n+1>1$.  
    \begin{enumerate}
        \item 
  If $p>n+1$, suppose that  $h$ is a solution to Eq. \eqref{eq-monge-ampere-Lp}, then there holds
    \begin{eqnarray}\label{thm-lower}
        \min\limits_{\C_{\theta}}h \geq c, 
    \end{eqnarray}
    and for any $\alpha \in (0, 1)$, 
    \begin{eqnarray}\label{thm-C2}
        \|h\|_{C^{3,\alpha}(\C_{\theta})}\leq C,
    \end{eqnarray}
   where the constants $c, C$ depend only on $n, p$ and $f$.
  \item  If $1<p<n+1$, suppose that $h$ is a positive, capillary even solution to Eq. \eqref{eq-monge-ampere-Lp}, then  \eqref{thm-lower} and \eqref{thm-C2} still hold. 
    \end{enumerate}
   Furthermore, if $n+1<p<n+2$, the rescaling solution 
   \begin{eqnarray*}
       \widetilde{h}\coloneqq \frac{h}{\min\limits_{\C_{\theta}}h}
   \end{eqnarray*}
satisfies
\begin{eqnarray}\label{re-est}
    \|\widetilde{h}\|_{C^{3,\alpha}(\C_{\theta})}\leq C',
\end{eqnarray}
where the constant $C'$ depends only on $n$ and $f$, but is independent of $p$.
\end{thm}

In order to prove Theorem \ref{thm priori est}, we begin by establishing the $C^0$ estimate.

\subsection{\texorpdfstring{$C^{0}$}{} estimate} 

In this subsection, we establish positive lower and upper bounds for a solution to Eq.~\eqref{eq-monge-ampere-Lp} with $p>1$ and $p\neq n+1$. We first consider the case $p>n+1$.
\begin{lem}\label{C0 pb}
    Let $p > n+1$ and $\theta\in (0, {\pi})$. Suppose  $h$ is a positive solution to Eq. \eqref{eq-monge-ampere-Lp}. Then,  if $\theta \in (0, \frac{\pi}{2}]$, there holds
    \begin{eqnarray}\label{up-low}
       \frac{(1-\cos\theta)^{p-n-1}}{ (\sin\theta)^{2p-2} } \min\limits_{  \C_{\theta}}f^{-1}\leq h^{p-n-1}\leq   \frac{(\sin\theta)^{2(p-n-1)}}{(1-\cos\theta)^{p-1} }\max\limits_{ C_{\theta}}f^{-1},
    \end{eqnarray}
    and if  $\theta\in (\frac{\pi}{2}, \pi)$, we obtain 
    we obtain
    \begin{eqnarray}\label{c0 est-1}
        \frac{(\sin\theta)^{2(p-n-1)}}{ (1-\cos\theta)^{p-1}} \min\limits_{\C_{\theta}}f^{-1}\leq h^{p-n-1}\leq \frac{(1-\cos\theta)^{p-n-1}}{  (\sin\theta)^{2(p-1)}}\max\limits_{\C_{\theta}}f^{-1}.
    \end{eqnarray}
    \end{lem}

\begin{proof}
Consider the capillary support function (see, e.g.,  \cite[Eq. (1.6)]{MWW-AIM})
    \begin{eqnarray*}
        u \coloneqq \ell^{-1} h,
    \end{eqnarray*}
by \eqref{eq-monge-ampere-Lp}, we know that $u$ satisfies 
\begin{eqnarray}\label{eq-monge-ampere-Lp-u}
\left\{\begin{array}{rcll}\vspace{2mm}\displaystyle
		\det\left( \ell \n^2 u+\cos \theta ( \n u \otimes e^T +e^T\otimes \n u)+  u \sigma\right)&=&  f(u\ell)^{p-1},& \quad    \hbox{ in } \C_{\theta},\\
			\n_\mu u &=& 0, &\quad  \hbox{ on } \p \C_\theta,
		\end{array}\right.
	\end{eqnarray}  
where we used the simple fact that $\n_{\mu}\ell =\cot\theta \ell$ on $\partial \C_{\theta}$.

Suppose  $u$ attains the maximum value at some point, say $\xi_{0}\in \C_{\theta}$. If $\xi_{0}\in \C_{\theta}\setminus \partial\C_{\theta}$, then 
\begin{eqnarray}\label{deri12}
    \n u(\xi_{0})=0\quad \text{and}\quad \n^{2}u(\xi_{0})  \leq 0.
\end{eqnarray}
If $\xi_{0}\in \partial\C_{\theta}$, the boundary condition in \eqref{eq-monge-ampere-Lp-u} implies that \eqref{deri12} still holds. In the following, we perform computations at $\xi_{0}$. Substituting \eqref{deri12} into \eqref{eq-monge-ampere-Lp-u}, we obtain
\begin{eqnarray*}
   f(u\ell)^{p-1}= \det\left( \ell \n^2 u+\cos \theta ( \n u \otimes e^T +e^T\otimes \n u)+  u \sigma\right) \leq u^{n},
\end{eqnarray*}
which implies 
\begin{eqnarray*}
    u^{p-1-n} (\xi_0)\leq f^{-1}\ell^{1-p}\leq\frac{1}{\min\limits_{\C_{\theta}}f \cdot (1-\cos\theta)^{p-1}}.
\end{eqnarray*}
Therefore, for all $\xi\in \C_{\theta}$, 
\begin{eqnarray}
    h^{p-n-1}(\xi)&=&\left(u\ell\right)^{p-n-1}(\xi)
    \leq  u^{p-1-n}(\xi_{0})\cdot \left(\max \limits_{\xi\in \C_{\theta}}\ell(\xi)\right)^{p-n-1}\notag \\
    &\leq& \frac{(\sin\theta)^{2(p-n-1)}}{\min\limits_{C_{\theta}}f \cdot (1-\cos\theta)^{p-1}}.\label{h max}
\end{eqnarray}
Similarly, there holds
\begin{eqnarray*}
    h^{p-n-1}(\xi)\geq \frac{(1-\cos\theta)^{p-n-1}}{\max\limits_{ \C_{\theta}}f \cdot (\sin\theta)^{2(p-1)}}.
\end{eqnarray*}
We derive that \eqref{up-low} holds and \eqref{c0 est-1} follows from a similar argument. This completes the proof.    
\end{proof}

Next, we deal with the case $1<p<n+1$. By using the capillary isoperimetric inequality (see, e.g., \cite{Maggi} or \cite{MWW24-IMRN}) and H\"older's inequality, we derive a positive upper bound of any capillary even solution to Eq. \eqref{eq-monge-ampere-Lp}. Subsequently, by taking into account the special structure of Eq. \eqref{eq-monge-ampere-Lp}, we can obtain a positive lower bound of the volume, which also implies the positive lower bound of the capillary even solution to Eq. \eqref{eq-monge-ampere-Lp}.

\begin{lem}\label{C0 ps}
  Let $1<p<n+1$ and $\theta\in (0, \pi)$. Suppose $h$ is a positive, capillary even solution to Eq. \eqref{eq-monge-ampere-Lp}, then there exist positive constants $C_1,C_2$ depending only on $n, p, \min\limits_{\C_{\theta}}f$ and $\max\limits_{\C_{\theta}}f$ such that
    \begin{eqnarray}\label{bi-bound-h}
        C_1\leq h\leq C_2.
    \end{eqnarray}    
\end{lem}

 \begin{proof}
 First, we derive the positive upper bound of $h$. Since $h$ is a capillary even function on $\C_{\theta}$, then $h$ satisfies $$\int_{\C_{\theta}}\<\xi, E_{\alpha}\>hd\sigma =0, \quad {\text{for~}}~ 1\leq \alpha\leq n.$$
 From \cite[Lemma~3.1]{MWW-AIM}, we have that the origin lies in the interior of $\widehat{\partial\S}$. Consequently,  there exists some	$R>0$ such that 
		\begin{eqnarray}\label{enclosed}
			\widehat	\S\subset \widehat{\C_{\theta,R}},
		\end{eqnarray}where   $$\C_{\theta,R}\coloneqq  \left\{\varsigma \in \ol{\RR^{n+1}_+} \big| ~~|\varsigma-\cos\theta R e|=R \right \}.$$
   Let $R_{0}$ denote the smallest positive constant $R$ satisfying \eqref{enclosed}.  It is clear that there exists $X_0\in \Sigma \cap R_0\C_{\theta } $.  Set $\widehat {X}_0\coloneqq \frac{X_0 }{R_0}\in \C_\theta$,  for any $\xi \in \C_\theta$, we have 
		\begin{eqnarray*}
			h(\xi) &= &\sup_{Y\in \S}\<\xi-\cos\theta e,Y\>\notag  \geq \max \left\{0, \<\xi-\cos\theta e,X_0\>  \right\} \notag \\&=&  \max \left\{0, R_0\<\xi-\cos\theta e,\widehat {X}_0\> \right\}.
   \end{eqnarray*}
   Integrating over $\xi\in\C_\theta$ yields 
   \begin{eqnarray}\label{hp}
       \int_{\C_{\theta}}h^{p}d\sigma\geq R_{0}^{p}\int_{\C_{\theta}}\left(\max \left\{0, \<\xi-\cos\theta e, \widehat{X}_{0}\> \right\}\right)^{p}d\sigma.
   \end{eqnarray}
   
   On the other hand, by multiplying $h$ on both sides of Eq. \eqref{eq-monge-ampere-Lp} and integrating over $\C_{\theta}$, we have
\begin{eqnarray}
  \int_{\C_{\theta}}f h ^{p}d\sigma&=&\int_{\C_{\theta}} h\det(\n^{2}h+h\sigma)d\sigma=\int_{\S} \<X, \nu\>dA \notag \\
  &=&(n+1)|\widehat{\S}|.\label{volume}
\end{eqnarray}
From the capillary isoperimetric inequality (see e.g., \cite[Theorem 19.21]{Maggi} or \cite[Eq. (1.7)]{MWW24-IMRN}), we obtain
		\begin{eqnarray*}
			c(n,\theta)	|\widehat{\S}|^{\frac{n}{n+1}}&\leq& |\S|-\cos\theta|\widehat{\p\S}|=\int_\S \left(1+\cos\theta \<\nu,e\>\right)dA\notag 
			\\&=&\int_{\C_\theta} \left(\sin^2\theta+\cos\theta\<\xi,e\>\right)\det(\n^{2}h+h\sigma) d \s,
		\end{eqnarray*}
        where $c(n,\theta)\coloneqq \frac{|\C_\theta|-\cos\theta |\widehat{\p \C_\theta}|}{|\widehat{\C_\theta}|^{\frac{n}{n+1}}}.$ 
\eqref{volume} and  H\" older's inequality together yield
  \begin{eqnarray*}
      \int_{\C_{\theta}}f h^{p}d\sigma &\leq& C(n,\theta)\left(\int_{\C_{\theta}}\det(\n^{2}h+h\sigma)d\sigma\right)^{\frac{n+1}{n}}\\
      &=&C(n,\theta)\left(\int_{\C_{\theta}}h^{p-1}fd\sigma\right)^{\frac{n+1}{n}}\\
      &\leq& C(n,\theta)\left(\int_{\C_{\theta}}fh^{p}d\sigma\right)^{\frac{p-1}{p}\frac{n+1}{n}}\left(\int_{\C_{\theta}}fd\sigma\right)^{\frac{1}{p}\frac{n+1}{n}},
  \end{eqnarray*}
 where $C(n,\theta)\coloneqq (n+1)\left(\frac 2 {c(n,\theta)} \right)^{\frac{n+1}{n}}$. Since $p<n+1$, we conclude that 
  \begin{eqnarray*}
      \int_{\C_{\theta}}fh^{p}d \s\leq C(n,\theta)^{\frac{np}{n+1-p}}\left(\int_{\C_{\theta}}fd\sigma\right)^{\frac{n+1}{n+1-p}},  \end{eqnarray*}
and further,
\begin{eqnarray*}
    \int_{\C_{\theta}}h^{p}d\sigma\leq \frac{C(n,\theta)^{\frac{np}{n+1-p}}}{\min\limits_{\C_{\theta}}f}\left(\int_{\C_{\theta}}fd\sigma\right)^{\frac{n+1}{n+1-p}}.
\end{eqnarray*}
Combining  with \eqref{hp}, we get
\begin{eqnarray*}
    R_{0}\leq \left(\int_{\C_{\theta}} \left(\max \left\{0, \<\xi-\cos\theta e, \widehat{X}_{0}\> \right\} \right)^{p}d\sigma\right)^{-\frac{1}{p}}\left(\frac{C(n,\theta)^{\frac{np}{n+1-p}}}{\min\limits_{\C_\theta}f} \left(\int_{\C_{\theta}}f d\sigma \right)^{\frac{n+1}{n+1-p}}\right)^{\frac{1}{p}},
\end{eqnarray*}
which implies the positive upper bound of $h$ in \eqref{bi-bound-h}.

Next, we establish the positive lower bound of $h$ in \eqref{bi-bound-h}. By adopting the argument as in  Lemma \ref{C0 pb}, from \eqref{h max} we obtain
\begin{eqnarray*}
    \max_{\C_\theta} h\geq \left(\min\limits_{\C_{\theta}}f \cdot (1-\cos\theta)^{p-1}\sin^{2(n+1-p)}\theta\right)^{\frac{1}{n+1-p}}.
\end{eqnarray*}
From \cite[Remark 2.3]{MWW-AIM}, we have
\begin{eqnarray}\label{lower of R0}
    R_{0}&=&\frac{\max\limits_{\C_\theta} h}{1+\cos\theta\<\nu(X_{0}), e\>}
    \geq \frac{1}{2}\max\limits_{\C_\theta} h \notag \\
    &\geq& \frac{1}{2}\left(\min\limits_{\C_{\theta}}f \cdot (1-\cos\theta)^{p-1}\sin^{2(n+1-p)}\theta\right)^{\frac{1}{n+1-p}}.
\end{eqnarray}
Together with \eqref{hp} and \eqref{volume},  it yields
\begin{eqnarray*}
    (n+1)|\widehat{\S}|&=&\int_{\C_{\theta}}f h^{p}d\sigma\geq \min\limits_{\C_{\theta}}f \int_{\C_{\theta}}h^{p}d\sigma\notag \\
    &\geq&\min\limits_{\C_{\theta}}f\cdot R_{0}^{p}\int_{\C_{\theta}}\left(\max \left\{0, \<\xi-\cos\theta e, \widehat{X}_{0}\> \right\}\right)^{p}d\sigma.
\end{eqnarray*}
Taking into account \eqref{lower of R0}, we derive that the volume $|\widehat{\S}|$ has a uniform positive lower bound. Additionally,  since there is a uniformly upper bound of $R_{0}$, it follows that the inner radius must have a positive lower bound, which in turn yields the lower bound of \eqref{bi-bound-h}. This completes the proof. 
    
\end{proof}

\subsection{\texorpdfstring{$C^{1}$}{} estimate}\label{sec3.2}
In this subsection, we derive the $C^{1}$ estimate, which follows from convexity and is independent of the specific equation. For the reader’s convenience, 
we include the proof here.
\begin{lem}\label{C1}
    Let $p>1$ and $\theta\in (0, \pi)$. Suppose $h$ is a positive solution to Eq. \eqref{eq-monge-ampere-Lp}, then there holds
    \begin{eqnarray}\label{gradient ps}
       \max\limits_{\C_{\theta}} |\n h|\leq (1+\cot^{2}\theta)^{\frac{1}{2}}\|h\|_{C^{0}(\C_{\theta})}.
    \end{eqnarray}
   
\end{lem}
\begin{proof}
    Consider the following  function
    \begin{eqnarray*}
        P\coloneqq |\n h|^{2}+h^{2}.
    \end{eqnarray*}
Suppose that $P$ attains its maximum value at some point, say $\xi_{0}\in \C_{\theta}$. If $\xi_{0}\in \C_{\theta}\setminus \partial\C_{\theta}$, by the maximal condition, we have
\begin{eqnarray*}
    0=\n_{e_i}P=2h_{k}h_{ki}+2hh_{i},\quad \text{for}~1\leq i\leq n.
\end{eqnarray*}
Together with the convexity of $h$, i.e., $h_{ij}+h\delta_{ij}$ is positive definite,  it follows $\n h(\xi_{0})=0$ and we conclude that \eqref{gradient ps} holds.

If $\xi_{0}\in \partial\C_{\theta}$, we choose an orthonormal frame $\{e_{i}\}_{i=1}^{n}$ around $\xi_{0}\in \partial\C_{\theta}$ such that $e_{n}=\mu$. From \cite[Proposition 2.8]{MWW-AIM}, we know that
\begin{eqnarray}\label{h_1n}
    h_{\alpha n}=0, ~~\text{for~any}~~ 1\leq \alpha\leq  n-1.
\end{eqnarray}
Using the maximal condition again, we have 
\begin{eqnarray*}
   0= \n_{e_{\alpha}} P=2\sum\limits_{i=1}^{n}h_{i}h_{i\alpha}+2hh_{\alpha},
\end{eqnarray*}
which implies 
\begin{eqnarray}\label{tangent}
    h_\alpha(\xi_0)=0.
\end{eqnarray}
Together with  \eqref{tangent} and $h_{n}=\cot\theta h$ on $\partial\C_{\theta}$, we get
\begin{eqnarray*}
    |\n h|^{2}(\xi_0)\leq (|\n h|^{2}+|h|^{2})(\xi_{0})=(|h_{n}|^{2}+h^{2})(\xi_{0})=(1+\cot^{2}\theta)h^{2}(\xi_{0}).
\end{eqnarray*}
Thus, the assertion \eqref{gradient ps} follows.
\end{proof}

We now move on to establish the logarithmic gradient estimate for solutions to Eq. \eqref{eq-monge-ampere-Lp}. This estimate is crucial for addressing the case $p = n+1$. Similar type estimate were established in \cite[Section 2]{GL2000}, \cite[Section 3]{HMS2004} and \cite[Section 3]{MWW-Quotient}.  Our choice of test functions is inspired by the ideas developed in \cite{MQ, MX}, in the context of boundary value problems.

\begin{lem}\label{lem kC1}
 Let $p\geq n+1$ and $\theta \in (0, \frac{\pi}{2})$. If $h$ is a positive solution to Eq. \eqref{eq-monge-ampere-Lp}, then there exists a  constant $C$ depending on $n, p, \min\limits_{\C_{\theta}}f$ and $\|f\|_{C^{1}(\C_{\theta})}$, such that 
\begin{eqnarray}\label{log-C1-est}
\max\limits_{\C_{\theta}} |\n \log h|\leq C.
  \end{eqnarray}
Furthermore, if $n+1\leq p\leq n+2$, the constant $C$ in \eqref{log-C1-est} is independent of $p$.
  \end{lem}

  \begin{proof}
Define $v\coloneqq \log h$,  from \eqref{eq-monge-ampere-Lp}, then $v$ satisfies 
\begin{eqnarray}\label{v equ}
\left\{
		\begin{array}{rcll}\vspace{2mm}\displaystyle
			\det \left(\n^2 v+\n v\otimes \n v+ \s \right)&= &e^{p_{0}v}f\coloneqq \widehat{f},& \quad    \hbox{ in } \C_{\theta},\\
			\n_\mu v &=& \cot\theta   , &\quad  \hbox{ on } \p \C_\theta,
		\end{array}\right.
	\end{eqnarray}  
 where $p_{0}\coloneqq p-n-1$.

 Let $N\coloneqq(0, \cdots, 0, 1-\cos\theta)\in \C_\theta$, we introduce the function $d\coloneqq\frac{1}{2\theta} d_N^2(\xi)$, where $d_N(\xi)$ is the geodesic distance function from $\xi$ to $N$ on $\C_\theta$. It is easy to see that $d$ is well-defined and smooth for all $\xi\in \C_\theta$. Moreover, from direct computation, $d$ satisfies $d=\frac{\theta}2$,  $\n d=\mu$ on $\partial\C_{\theta}$ and 
\begin{eqnarray}\label{positive-define}
    (\n^{2}d) \geq \min\left\{ \frac{1}{\theta}, {\cot\theta}  \right\}\sigma\eqqcolon c_{0}\sigma~~ {\rm{in}}~~ \C_{\theta}. \end{eqnarray} 
    To proceed, we assume that $|\nabla v|$ is sufficiently large; otherwise, the conclusion follows directly. Consider the  function 
\begin{eqnarray*}
\Phi\coloneqq \log|\nabla w|^{2}+\varphi(d),
\end{eqnarray*}
where  $w\coloneqq v-\cot\theta d$, and  $\varphi(d)\coloneqq e^{sd}$ for some constant $s\in \RR$ to be determined later. 
 
Assume that $\Phi$ attains its maximum at some point, say $\xi_0 \in \C_\theta$. We divide the proof into two cases: either  $\xi_0\in \p \C_\theta$ or $\xi_0\in \C_\theta\setminus \p \C_\theta$.

 \
 
\noindent {\it Case 1}. $\xi_{0}\in \partial \C_{\theta}$. Let $\{e_{i}\}_{i=1}^{n}$ be an orthonormal frame around $\xi_0$ such that $e_{n}=\mu$.  The maximal condition implies
 \begin{eqnarray}\label{Phi-normal}
     0\leq \n_{n}\Phi(\xi_{0})=\frac{ w_{k} w_{kn}}{|\n  w|^{2}}+\varphi'\n_{n}d.
 \end{eqnarray}
 Notice that \begin{eqnarray}\label{w-n-normal}
     w_{n}=v_{n}-\cot\theta d_{n}=0,
 \end{eqnarray} then $|\n w|=\sum\limits_{\a=1}^{n-1} w_\a^2$. From the Gauss-Weingarten equation of $\partial\C_{\theta}\subset \C_{\theta}$,
 \begin{eqnarray}\label{w-alpha-n}
      w_{\alpha n}=\n_{n}( w_{\alpha})-\n  w(\n_{e_{\a}}e_{n})=-\cot\theta  w_{\alpha}.
     \end{eqnarray}
Inserting \eqref{w-n-normal} and \eqref{w-alpha-n} into \eqref{Phi-normal} yields
    \begin{eqnarray*} 
 0\leq \n_{n}\Phi(\xi_{0})=\frac{\sum\limits_{\a=1}^{n-1} w_{\a} w_{\a n}}{|\n  w|^{2}}+\varphi'\n_{n}d=-\cot\theta+se^{sd(\xi_0)}<0 ,    \end{eqnarray*}
where the last inequality follows if we choose $s\leq \s_0$ with $\s_0>0$ satisfying 
\begin{eqnarray}\label{s-choice-1}
    \s_0e^{\frac{\theta \s_{0}} 2} <\cot\theta.
\end{eqnarray}This yields a contradiction. Hence $\xi_0\in \C_\theta \setminus\p \C_\theta$. 

\

\noindent{\it Case 2.} $\xi_{0}\in \C_{\theta}\setminus \partial \C_{\theta}$. 
We reformulate \eqref{v equ} by denoting 
\begin{eqnarray}\label{eq-log-B}
    G(B)\coloneqq \log \det B=\log \widehat{f}
\end{eqnarray} where $B\coloneqq (B_{ij})=(v_{ij}+v_{i}v_{j}+\delta_{ij})$. Now we choose an orthonormal frame  $\{e_{i}\}_{i=1}^{n}$ around $\xi_{0}$, such that  $B_{ij}$ is diagonal at $\xi_0$.  Hence $G^{ij}\coloneqq \frac{\partial G(B)}{\partial B_{ij}}$ is also diagonal at $\xi_0$.   The maximal value condition at $\xi_0$ implies 
\begin{eqnarray}\label{Phi-1}
    0=\Phi_{i}=\frac{|\n  w|^{2}_{i}}{|\n  w|^{2}}+\varphi'd_{i},
\end{eqnarray}
and 
\begin{eqnarray}\label{two-der}
    0&\geq &G^{ij} \Phi_{ij}=G^{ij}\left(\frac{|\n  w|^{2}_{ij}}{|\n  w|^{2}}-\frac{|\n w|^{2}_{i}|\n  w|^{2}_{j}}{|\n  w|^{4}}+\varphi''d_{i}d_{j}+\varphi'd_{ij}\right)\notag \\
    &=&\frac{2G^{ij} w_{k} w_{kij}+2G^{ij} w_{ki} w_{kj}}{|\n  w|^{2}}+G^{ij}\left[(\varphi''-(\varphi')^{2})d_{i}d_{j}+\varphi'd_{ij}\right].
\end{eqnarray}
For the standard metric on spherical cap $\C_\theta$, the commutator formulae gives
		\begin{eqnarray}\label{3-third comm}
			v_{kij}=v_{ijk}+v_k\d_{ij}-v_j\d_{ki}. 
		\end{eqnarray}
  Taking the first derivatives of Eq. \eqref{eq-log-B} in the $e_{k}$ direction, we have 
  \begin{eqnarray}
      G^{ij}(v_{ijk}+2v_{ki}v_{j})=(\log \widehat{f})_{k}. \label{one-der}
  \end{eqnarray}
 Inserting \eqref{3-third comm} and \eqref{one-der} into \eqref{two-der}, we obtain 
  \begin{eqnarray}
      0&\geq &\frac{2G^{ij} w_{k}(v_{kij}-\cot\theta d_{kij})+2G^{ij} w_{ki} w_{kj}}{|\n  w|^{2}}+G^{ij}\left[(\varphi''-(\varphi')^{2})d_{i}d_{j}+\varphi'd_{ij}\right]\notag \\
      &=&\frac{2 w_{k}\left[(\log \widehat{f})_{k}-2G^{ij}v_{ki}v_{j}+v_{k}\sum\limits_{i=1}^{n}G^{ii}-G^{ij}v_{j}\delta_{ki}-\cot\theta d_{kij}\right]}{|\n  w|^{2}}\notag \\
      &&+\frac{2G^{ij} w_{ki} w_{kj}}{|\n  w|^{2}}+G^{ij}\left[(\varphi''-(\varphi')^{2})d_{i}d_{j}+\varphi'd_{ij}\right]\notag  \\
      &\geq & -\frac{C}{|\n  w|} \left(\sum\limits_{i=1}^{n}G^{ii}+1 \right)+\frac{2G^{ij} w_{ki} w_{kj}-4G^{ij}v_{ki}v_{j} w_{k}}{|\n  w|^{2}}\notag \\
      &&+G^{ij}\left[(\varphi''-(\varphi')^{2})d_{i}d_{j}+\varphi'd_{ij}\right],\label{log-C1-1}
  \end{eqnarray}
where the constant $C$ depends on $n$ and $\|\log f\|_{C^{1}(\C_{\theta})}$. To proceed, we denote $\mathcal{I} \coloneqq  2G^{ij} w_{ki} w_{kj}-4G^{ij}v_{ki} w_{k}v_{j}$. A straightforward computation yields
  \begin{eqnarray*}
      \mathcal{I}&=&-4G^{ij} \left(B_{ki}-v_{k}v_{i}-\delta_{ki} \right)v_{j} w_{k}+2G^{ij} \left(B_{ki}-v_{k}v_{i}-\delta_{ki}-\cot\theta d_{ki} \right)\notag \\
      &&\cdot \left(B_{kj}-v_{k}v_{j}-\delta_{kj}-\cot\theta d_{kj} \right)\notag \\
      &= &2G^{ij}B_{ki}B_{kj}-4G^{ij}B_{ki}v_{j} w_{k}-4G^{ij}B_{ki}(v_{k}v_{j}+\delta_{kj}+\cot\theta d_{kj})\notag\\
      &&+2G^{ij}(v_{k}v_{j}+\delta_{kj}+\cot\theta d_{kj})(v_{k}v_{i}+\delta_{ki}+\cot\theta d_{ki})\notag \\
      &&+4G^{ij}v_{i}v_{j}v_{k} w_{k}+4G^{kj}v_{j} w_{k}\notag  \\
      &\geq &2\sum\limits_{i=1}^{n}B_{ii}+ 2|\n v|^{2}G^{ij}v_{i}v_{j}+4G^{ij}v_{i}v_{j}v_{k} w_{k}+4G^{ij}v_{i}v_{j}\notag  \\
      &&-C|\n v| \left(\sum\limits_{i=1}^{n} G^{ii}+1 \right)-8|\n v|^{2}.
  \end{eqnarray*}
Next, we {\bf claim} that:
  \begin{eqnarray}
      2\sum\limits_{i=1}^{n}B_{ii}+4G^{ij}v_{i}v_{j}v_{k} w_{k}+2G^{ij}v_{i}v_{j}|\n v|^{2}-8|\n v|^{2}\geq -C|\n v|. \label{claim}
  \end{eqnarray}
 In fact, from \eqref{Phi-1} we have 
\begin{eqnarray*}
    2 w_{k} \left(B_{ki}-v_{k}v_{i}-\delta_{ki}-\cot\theta d_{ki} \right)=2 w_{k} w_{ki}=-\varphi'd_{i}|\n  w|^{2}, 
\end{eqnarray*}
which implies for all fixed $1\leq i\leq n$, 
\begin{eqnarray}
    2 w_{i}B_{ii} =-\varphi'd_{i}|\n  w|^{2}+2v_{i}v_{k} w_{k}+2 w_{k}(\delta_{ki}+\cot\theta d_{ki}).\label{Bii}
\end{eqnarray}
Without loss of generality, we assume that $v_{1}^{2}\geq v_{2}^{2}\geq \cdots \geq v_{n}^{2}$
, then $v_1$ is also sufficiently large by $v_{1}^{2}\geq \frac{|\n v|^{2}}{n}$, which follows $w_1\neq 0$. From \eqref{Bii} and $v_i=w_i+\cot\theta d_i$, there holds
\begin{eqnarray}
    B_{11}&=&\frac{1}{2 w_{1}}\left[-\varphi'd_{1}|\n  w|^{2}+2 w_{k}v_{k}v_{1}+2 w_{k}(\delta_{k1}+\cot\theta d_{k1})
    \right]\notag \\
    &=&\frac{ w_{k}v_{k}v_{1}}{ w_{1}}-\frac{|\n  w|^{2}}{2 w_{1}}\varphi'd_{1}+O(1)\notag \\
    &=&|\n v|^{2}+O(|\n v|).\label{B11}
\end{eqnarray}
On the other hand, for all $2\leq i\leq n$, if $ w_{i}\neq 0$, then using \eqref{Bii} again yields
\begin{eqnarray}
    B_{ii}=\frac{1}{2 w_{i}}\left[-\varphi'd_{i}|\n  w|^{2}+2 w_{k}v_{k}v_{i}+2 w_{k}(\delta_{ki}+\cot\theta d_{ki})\right].\label{Bii-1}
\end{eqnarray}
Denote $$\mathcal{S}\coloneqq \left\{i: i\geq 2~{\rm{ and }}~v_{i}^{2}(\xi_0)>  \cot^{2}\theta \|\n d\|_{C^{0}(\C_{\theta})}\right\}.$$ If $\mathcal{S}=\emptyset$, then it is easy to see that $|v_{1}|=|\n v|+O(1)$. In this case, together with \eqref{B11}, we obtain
\begin{eqnarray}
4G^{ij}v_{i}v_{j}v_{k} w_{k}+2G^{ij}v_{i}v_{j}|\n v|^{2}&\geq&4G^{11}v_{1}^{2}w_{k}v_{k}+2G^{11}v_{1}^{2}|\n v|^{2}\notag \\
&=&\frac{1}{B_{11}}\left(4v_{1}^{2} w_{k}v_{k}+2v_{1}^{2}|\n v|^{2}\right) \notag \\
&=&6|\n v|^{2}+O(1). \label{empty}
\end{eqnarray}
Substituting \eqref{B11} and \eqref{empty} into the left-hand side of \eqref{claim}, we prove the claim holds. We now turn to the case $\mathcal{S} \neq \emptyset$. For each $i \in \mathcal{S}$, we have $ w_i \neq 0$. From \eqref{Bii-1}, it follows that 
\begin{eqnarray*}
   && \sum\limits_{i\in \mathcal{S}}4G^{ij}v_{i}v_{j}v_{k} w_{k}+2G^{ij}v_{i}v_{j}|\n v|^{2}\notag  =4\sum\limits_{i\in \mathcal{S}}\frac{v_{i}^{2}v_{k}  w_{k}}{B_{ii}}+2\sum\limits_{i\in \mathcal{S}}\frac{v_{i}^{2}|\n v|^{2}}{B_{ii}}\notag \\
    &=&\sum\limits_{i\in \mathcal{S}}\frac{v_{i}^{2} w_{i}\left(4v_{k} w_{k}+2|\n v|^{2}\right)}{-\frac{1}{2}\varphi'd_{i}|\n  w|^{2}+ w_{k}v_{k}v_{i}+ w_{k}(\delta_{ki}+\cot\theta d_{ki})}
.
\end{eqnarray*}
 Since \begin{eqnarray*}
    \sum_{i\in \mathcal S} v_i w_i=\sum_k w_k v_k +O(1)=|\n v|^{2}+O(|\n v|),
\end{eqnarray*} 
we have 
\begin{eqnarray*}
&&\sum\limits_{i\in \mathcal{S}}\frac{v_{i}^{2} w_{i}}{-\frac{1}{2}\varphi'd_{i}|\n  w|^{2}+ w_{k}v_{k}v_{i}+ w_{k}(\delta_{ki}+\cot\theta d_{ki})}\notag\\
&=&\sum\limits_{i\in \mathcal{S}}\frac{v_{i}^{2} w_{i}}{ w_{k}v_{k}v_{i}}\cdot  \left[{1-\frac{\varphi'd_{i}|\n  w|^{2}}{2 w_{k}v_{k}v_{i}}+\frac{ w_{k}(\delta_{ki}+\cot\theta d_{ki})}{ w_{k}v_{k}v_{i}}} \right]^{-1} \notag \\
&\geq &\sum\limits_{i\in \mathcal{S}}\frac{v_{i}^{2} w_{i}}{ w_{k}v_{k}v_{i}} \left[1+\frac{\varphi'd_{i}|\n  w|^{2}}{2 w_{k}v_{k}v_{i}}-\frac{ w_{k}(\delta_{ki}+\cot\theta d_{ki})}{ w_{k}v_{k}}\right]\notag \\
&\geq& \sum\limits_{i\in \mathcal{S}}\left(\frac{v_{i} w_{i}}{ w_{k}v_{k}}+\frac{ w_{i}\varphi'd_{i}|\n  w|^{2}}{2( w_{k}v_{k})^{2}}\right)+O(|\n v|^{-1})=1+O(|\n v|^{-1}).
\end{eqnarray*}
And hence we obtain 
\begin{eqnarray*}
    \sum\limits_{i\in \mathcal{S}}4G^{ij}v_{i}v_{j}v_{k} w_{k}+2G^{ij}v_{i}v_{j}|\n v|^{2}=6|\n v|^{2}+O(|\n v|),
\end{eqnarray*}
which, together with \eqref{B11}, yields claim \eqref{claim}. 

Now inserting \eqref{claim} into \eqref{log-C1-1}, we derive
\begin{eqnarray}
    0&\geq & -\frac{C \left(\sum\limits_{i=1}^{n}G^{ii}+1 \right)}{|\n  w|}+\frac{4G^{ij}v_{i}v_{j}-C|\n v|\left(\sum\limits_{i=1}^{n}G^{ii}+1 \right)}{|\n  w|^{2}}\notag \\
    &&+G^{ij}\left[(\varphi''-(\varphi')^{2})d_{i}d_{j}+\varphi'd_{ij}\right]\notag \\
    &\geq& -\frac{C \left(\sum\limits_{i=1}^{n}G^{ii}+1\right)}{|\n  w|}
    +G^{ij}\left[(\varphi''-(\varphi')^{2})d_{i}d_{j}+\varphi'd_{ij}\right].\label{3.31}
\end{eqnarray}
Since 
\begin{eqnarray*}
      \left(\varphi''-(\varphi')^{2}\right)G^{ij}d_{i}d_{j}&=&\theta^{-2}s^{2}(e^{sd}-e^{2sd}) d_{N}^{2}G^{ij}(d_{N})_{i}(d_{N})_{j}  \\
      &\geq& -2  s^{2}e^{2sd}\sum\limits_{i=1}^{n}G^{ii},  
\end{eqnarray*}
together with \eqref{positive-define} and \eqref{3.31}, we obtain 
\begin{eqnarray*}
    0&\geq & \left(c_{0}se^{sd}-2 s^{2}e^{2s d} \right)\sum\limits_{i=1}^{n}G^{ii}-\frac{C}{|\n  w|}\left(\sum\limits_{i=1}^{n}G^{ii}+1\right).
\end{eqnarray*}
Choose a small positive constant $s\leq \s_1$ with $\s_1$ satisfying \begin{eqnarray}\label{s-choice-2}
    \frac{c_{0}}{2}-2\s_1 e^{ \frac{\theta \sigma_{1}}2}>0,
\end{eqnarray} where $c_0$ is given by \eqref{positive-define} and is positive since $\theta<\frac \pi 2$. Then it follows 
\begin{eqnarray}
    0&\geq& \frac{1}{2}c_{0}se^{sd}\sum\limits_{i=1}^{n}G^{ii}-\frac{C}{|\n  w|} \left(\sum\limits_{i=1}^{n}G^{ii}+1 \right). \label{end}
\end{eqnarray}
Using the arithmetic-geometric inequality and \eqref{up-low}, we get
\begin{eqnarray}\label{sum-lower}
    \sum\limits_{i=1}^{n}G^{ii}=\sum\limits_{i=1}^{n}\frac{1}{B_{ii}}\geq n(\det B)^{-\frac{1}{n}}=n(h^{p-n-1}f)^{-\frac{1}{n}}\geq c_{1},
\end{eqnarray}for some positive constant $c_1$. Inserting \eqref{sum-lower} into \eqref{end}, we conclude that $$|\n w|\leq C,$$ which in turn implies $\max\limits_{\C_{\theta}}|\n v|\leq C$. 

\

In conclusion, we can choose $s = \min \{\sigma_0, \sigma_1\}$ such that the conditions \eqref{s-choice-1} and \eqref{s-choice-2} are satisfied. Therefore, we complete the proof.
  \end{proof}

\subsection{\texorpdfstring{$C^{2}$}{C2} estimate}
In this subsection, we establish the a priori $C^2$ estimate for the positive solution of Eq. \eqref{eq-monge-ampere-Lp}. 
\begin{lem}\label{lem c2 boundary}
    Let $p>1$ and $\theta\in (0,\frac{\pi}{2})$. Suppose  that $h$ is a positive  solution to  Eq. \eqref{eq-monge-ampere-Lp}, and denote $M\coloneqq \sup\limits_{\partial \C_{\theta}}|\n^{2}h(\mu, \mu)|$. If $p> 1$, then there holds
    \begin{eqnarray}
        \max\limits_{\C_{\theta}}|\n^{2}h|\leq M+C, \label{nor-1}
    \end{eqnarray}
    where $C>0$ depends  on $n, p, \min\limits_{\C_{\theta}}f, \min\limits_{\C_{\theta}}h,  \|f\|_{C^{2}(\C_{\theta})}$ and $\|h\|_{C^{0}(\C_{\theta})}$. Furthermore, if $n+1\leq p\leq n+2$, the constant $C$ is independent of $p$. 
\end{lem}
\begin{proof}
    	We consider the function 
		\begin{eqnarray*}
			P(\xi, \Xi)\coloneqq \n^{2}h(\Xi, \Xi)+h(\xi),
		\end{eqnarray*}
		for $\xi\in \C_{\theta}$ and a unit vector  $\Xi\in T_{\xi  }\C_{\theta}$.
		Suppose that $P$ attains its maximum value at some   point $\xi_{0}\in \C_{\theta} $ and some unit vector $\Xi_{0}\in T_{\xi_{0}}\C_{\theta}$. 
Again, we divide the proof into two cases: either  $\xi_0\in \p \C_\theta$ or $\xi_0\in \C_\theta\setminus \p \C_\theta$.

\
       
           \noindent \textit{Case 1.} $\xi_0\in \C_\theta \setminus\p \C_\theta$. In this case, we choose an orthonormal frame  $\{e_{i}\}_{i=1}^{n}$ around $\xi_{0}$, such that 
		$(A_{ij})=(h_{ij}+h\d_{ij})$ is diagonal and $\Xi_0=e_1$.
		Denote \begin{eqnarray}\label{eq3}
			F(A)\coloneqq \left(\det(A)\right)^{\frac{1}{n}}=h^{\frac{p-1}{n}}f^{\frac{1}{n}}\eqqcolon h^{a}\widetilde{f},
		\end{eqnarray}  and
		\begin{eqnarray*}
			F^{ij}\coloneqq \frac{\partial F(A)}{\partial A_{ij}},
			\quad  F^{ij,kl}\coloneqq \frac{\partial^{2}F(A)}{\partial A_{ij}\partial A_{kl}}.
		\end{eqnarray*}
	The homogeneity of $F$ implies
		\begin{eqnarray}\label{sum-1}
			F^{ij}h_{ij}=F^{ij}(A_{ij}-h \d_{ij})=h^{a}\widetilde{f}-h \sum\limits_{i=1}^{n}F^{ii}.
		\end{eqnarray} 
By	taking twice covariant derivatives in the $e_1$ direction to Eq. \eqref{eq3}, and utilizing the concavity of $F$, we obtain
		\begin{eqnarray}\label{tw0 deri}
			F^{ij}A_{ij11}&=&-F^{ij,kl}A_{ij1}A_{kl1}+(h^{a}\widetilde{f})_{11}\notag  \\
   &\geq& (h^{a})_{11}\widetilde{f}+2ah^{a-1}h_{1}\widetilde{f}_{1}+h^{a}\widetilde{f}_{11}.		\end{eqnarray}
From $p>1$, then $a=\frac{p-1}{n}>0$. To proceed, we assume $A_{11}>0$, otherwise, the conclusion follows. Then
\begin{eqnarray}\label{11-der}
    (h^{a})_{11}=ah^{a-1}(A_{11}-h)+a(a-1)h^{a-2}h_{1}^{2}\geq -ah^{a}+a(a-1)h^{a-2}h_{1}^{2}.
\end{eqnarray}
Inserting  \eqref{11-der} into \eqref{tw0 deri}, we get 
\begin{eqnarray}
    F^{ij}A_{ij11}\geq h^{a}(\widetilde{f}_{11}-a\widetilde{f})+a(a-1)\widetilde{f}h^{a-2}h_{1}^{2}-2ah^{a-1}|\n h||\n \widetilde f|. \label{impor-1}
\end{eqnarray}
		For the standard metric on the spherical cap $\C_\theta$,  we have the commutator formula 
	\begin{eqnarray*}
h_{klij}=h_{ijkl}+2h_{kl}\d_{ij}-2h_{ij}\d_{kl}+h_{li}\d_{kj}-h_{kj}\d_{il},
		\end{eqnarray*} 
from which we have 
		\begin{eqnarray}\label{change}
			F^{ij}h_{11 ij}&=& F^{ij}\left(h_{ij11}+2h_{11}\delta_{ij} -2 h_{ij}+h_{1i}\d_{1j} -h_{1j}\d_{i1}\right)%
		\notag	\\	
   &=&F^{ij}h_{ij11}+2h_{11}\sum\limits_{i=1}^{n}F^{ii}-2F^{ij}h_{ij}. 
\end{eqnarray}
From the arithmetic-geometric mean inequality, we have 
 		\begin{eqnarray}\label{lower bound}
			\sum\limits_{i=1}^{n}F^{ii}=\frac{1}{n} (\det A)^{\frac{1}{n}-1} \sum_{i=1}^n \frac{\partial \det A}{\partial A_{ii}}\geq 1.
		\end{eqnarray} 
  Combining  \eqref{sum-1},  \eqref{impor-1}, \eqref{change} and \eqref{lower bound}, we have 
\begin{eqnarray*}
    0&\geq &F^{ij}P_{ij}=F^{ij}h_{11iij}+F^{ij}h_{ij}\notag \\
    &=&F^{ij}(A_{ij11}-h_{11}\delta_{ij})+2h_{11}\sum\limits_{i=1}^{n}F^{ii}-F^{ij}h_{ij}\notag \\
    &\geq & (h_{11}+h)+ h^{a}\left[\widetilde{f}_{11}-(1+a)\widetilde{f}\right]
    +a(a-1)\widetilde{f}h^{a-2}h_{1}^{2}\notag \\
    &&-2ah^{a-1}|\n h||\n \widetilde f|.
\end{eqnarray*}
Together with Lemma \ref{C1}, for all $p>1$, we conclude that 
		\begin{eqnarray}\label{11}
			h_{11}(\xi_0)\leq C,
		\end{eqnarray}
where the positive constant $C$ depends on $n, p, \min\limits_{\C_{\theta}}f, \min\limits_{\C_{\theta}}h$, $ \|f\|_{C^{2}(\C_{\theta})}$ and $\|h\|_{C^{0}(\C_{\theta})}$. 
Furthermore, if $n+1\leq p\leq n+2$, we have $a\in 
 \left[1,\frac{n+1}n\right]$, which implies that the constant $C$ is independent of $p$.
  
	\

 \noindent\textit{Case 2.} $ \xi_{0} \in \partial \C_{\theta}$. In this case, we can follow the same argument as in \cite[Proof of Lemma 3.3, Case 2]{MWW-AIM} to obtain 
		\begin{eqnarray}\label{1n}
			\n^2	h(\Xi_0,\Xi_0)  \leq  |h_{\mu\mu}|(\xi_{0})+2 \|h\|_{C^0(\C_{\theta})}.
		\end{eqnarray}

Finally, combining \eqref{11} and \eqref{1n}, we conclude that \eqref{nor-1} holds.
\end{proof}

Next, we establish the double normal derivative estimate of $h$ on the boundary, which together with Lemma \ref{c2 est} yields the a priori $C^2$ estimate for the case $p>1$ and $p\neq n+1$.
	
	To begin with, we introduce an important function 
	\begin{eqnarray*}
		\zeta(\xi)\coloneqq e^{- d_{\p\C_\theta}(\xi)}-1,
	\end{eqnarray*}where $d_{\p\C_\theta}(\xi)\coloneqq \text{dist}(\xi,\p\C_\theta)$ is the geodesic distance function to $\p\C_\theta$ from $\xi\in \C_\theta$. 
	The function $\zeta(\xi)$ has been used in \cite[Lemma~3.1]{Gb99} and  \cite[Lemma~3.4]{MWW-AIM}. Notice that $\zeta$ satisfies 
    $$\zeta|_{\partial \C_{\theta}}=0, \quad \text{and} \quad \n \zeta|_{\partial \C_{\theta}}=\mu, $$
and  there exists a small constant $\delta_{0}>0$ such that
	\begin{eqnarray}\label{hessian of zeta}
		(	\n^{2}_{ij}\zeta) \geq \frac{1}{2} \min\{\cot\theta,1\}  \s,\quad {\text{in}}\quad \Omega_{\delta_{0}}.
	\end{eqnarray}
	where $\Omega_{\delta_{0}}\coloneqq \{\xi\in \C_{\theta}: d_{\partial\C_{\theta}}(\xi)\leq \delta_{0}\}$.

	\begin{lem}\label{c2 est}
		Let $p>1$ and $\theta\in (0, \frac{\pi}{2})$.  Suppose $h$ is a positive solution to Eq. \eqref{eq-monge-ampere-Lp}, then
		\begin{eqnarray*}
			\max\limits_{\C_{\theta}}|\n^{2}h|\leq C,
		\end{eqnarray*}
		where the positive constant $C$ depends only on $n, p, \min\limits_{\C_{\theta}}f, \min\limits_{\C_{\theta}}h, \|f\|_{C^{2}(\C_{\theta})}$ and  $\|h\|_{C^0(\C_{\theta})}$. Furthermore, if $n+1\leq p\leq n+2$,  the constant $C$ is independent of $p$.

	\end{lem}
	\begin{proof}
We consider an auxiliary function
		\begin{eqnarray*}
			Q(\xi)\coloneqq \<\n h, \n \zeta\>-\left(L+\frac{1}{2}M\right)\zeta(\xi)-\cot\theta h(\xi), \quad  \xi\in \Omega_{\delta_{0}},
		\end{eqnarray*}
		where	$L$ is a positive constant that will be determined later. We claim that $Q\ge 0$  in $\Omega_{\delta_0}$. 
        
        Let $\xi_0$ be a minimum point of  $Q$ in $\overline \Omega_{\delta_0}$.
We first show that $\xi_0\in \partial \Omega_{\delta_0}$.
		Assume by contradiction that $\xi_{0}\in \left(\Omega_{\delta_{0}}\setminus \partial \Omega_{\delta_{0}}\right)$. By choosing an orthonormal frame  $\{e_{i}\}_{i=1}^{n}$ around $\xi_{0}$  such that  $(A_{ij})$ is diagonal at $\xi_{0}$. From \eqref{3-third comm}, \eqref{hessian of zeta}, at $\xi_0$, we obtain
		\begin{eqnarray*}
			0&\leq & 	 F^{ij}Q_{ij} 
			\\&=&F^{ij}h_{kij}\zeta_{k}+F^{ij}h_{k}\zeta_{kij}+2F^{ij}h_{ki}\zeta_{kj}-\left(L+\frac{1}{2}M\right)F^{ii}\zeta_{ii}-\cot\theta F^{ii}h_{ii}
		\\&=&F^{ii}\left(A_{iik}-h_{i}\delta_{ki}
			\right)\zeta_{k}+2F^{ii}(A_{ii}-h)\zeta_{ii}-\left(L+\frac{1}{2}M\right)F^{ii}\zeta_{ii}\\
        &&+F^{ii}h_{k}\zeta_{kii}-\cot\theta F^{ii}h_{ii}\\
         &=&\zeta_{k} (h^{a}\widetilde f)_{k}-\sum\limits_{i=1}^{n}F^{ii}h_{i}\zeta_{i}+2F^{ii}A_{ii}\zeta_{ii}-2hF^{ii}\zeta_{ii}+F^{ii}h_k\zeta_{kii}-\cot\theta F^{ii}h_{ii}
			\\
   && -\left(L+\frac{1}{2}M\right)F^{ii}\zeta_{ii}\\
     &\leq  & C_{1}\left(h+|\n h|\right)\left(h^{a-1}+\sum\limits_{i=1}^{n}F^{ii}\right)-\frac{1}{2}\left(L+\frac{1}{2}M\right)\min\{\cot\theta,1\}\sum\limits_{i=1}^{n}F^{ii},
		\end{eqnarray*}
		where the constant $C_{1}$ depends on $a,  \min\limits_{\C_{\theta}}f$ and    $\|f\|_{C^1(\C_{\theta})}$.
		In view of \eqref{lower bound} and \eqref{gradient ps},
		it follows 
		\begin{eqnarray*}
			C_{1}(h+|\n h|)\left(h^{a-1}+\sum\limits_{i=1}^{n}F^{ii}\right)-\frac{1}{2}\left(L+\frac{1}{2}M\right)\min\{\cot\theta,1\}\sum\limits_{i=1}^{n}F^{ii}<0,
		\end{eqnarray*} 
		if we choose $L$ as
		\begin{eqnarray}\label{chosen of A}
			L&\coloneqq &\frac{4C_{1}\left[1+(1+\cot^{2}\theta)^{\frac{1}{2}}\right]\left(\|h\|_{C^{0}(\C_{\theta})}^{\frac{p-1}{n}}+\|h\|_{C^{0}(\C_{\theta})}\right)}{\min\{\cot\theta,1\}}\notag \\
   &&+\frac{1}{1-e^{-\delta_{0}}}\max\limits_{   \C_{\theta}} \left(|\n h|+\cot\theta  h \right). 
		\end{eqnarray} 
		   This  contradicts  $F^{ij}Q_{ij} \geq 0$ at $\xi_0$, which follows $\xi_0\in \p \O_{\d_0}$. Next, we proceed with two cases:
		\begin{enumerate}
		    \item 
			If  $\xi_0\in  \p \O_{\d_0} \cap \partial \C_{\theta}$, it is easy to observe that $Q(\xi_0)=0$. 
		
	\item 	If $\xi_0\in \partial \Omega_{\delta_{0}}\setminus  \partial \C_{\theta} $, we have $d(\xi_0)=\delta_0$, and from our choice of $L$ in \eqref{chosen of A},
		\begin{eqnarray*}
			Q(\xi)\geq -|\n h|+L(1-e^{-\delta_{0}})-\cot\theta h>0.
		\end{eqnarray*}
		\end{enumerate}	
Consequently, 
		\begin{eqnarray*}
			Q(\xi)\geq 0,\quad{\text{in}}\quad \Omega_{\delta_{0}}.
		\end{eqnarray*} 
		
		\
		
		Now we are ready to obtain the double normal second derivative estimate of $h$. 
		Assume  $h_{\mu\mu}(\eta_{0})\coloneqq \sup\limits_{\partial \C_{\theta}}h_{\mu\mu}>0$ for some $\eta_0\in \p \C_\theta$. In view of \eqref{h_1n}, \eqref{gradient ps} and $Q\equiv 0$ on $\p \C_\theta$, we have
		\begin{eqnarray*}
			0&\geq &Q_{\mu}(\eta_{0})\\
			&\geq &(h_{k\mu}\zeta_{k}+h_{k}\zeta_{k\mu})-\left(L+\frac{1}{2}M\right)\zeta_{\mu}-\cot\theta h_{\mu}\\
			&=& h_{\mu\mu}(\eta_{0})-\left(L+\frac{1}{2}M\right)-h_{k}\zeta_{k\mu}-\cot^{2}\theta h,
		\end{eqnarray*}
		which yields  
		\begin{eqnarray}\label{sup estimate}
			\max\limits_{\partial \C_{\theta}}h_{\mu\mu}\leq L+C_{2}\|h\|_{C^{0}(\C_{\theta})}+\frac{1}{2}M,
		\end{eqnarray}
  where $C_{2}$ is a uniform constant. 
		Similarly, we consider an auxiliary function as
		\begin{eqnarray*}
			\ov{Q}(\xi)\coloneqq \<\n h, \n \zeta\>+\left(\bar{L}+\frac{1}{2}M\right)\zeta(\xi)-\cot\theta h, \quad \xi\in \Omega_{\delta_{0}},
		\end{eqnarray*}
		where 	$\bar{L}>0$ is a positive constant to be determined later. Adapting the similar argument as above, we get $$\ov{Q}(\xi)\leq 0 ~\text{ in }   \Omega_{\delta_{0}},$$ and furthermore
		\begin{eqnarray}\label{inf estimate}
			\min\limits_{\partial \C_{\theta}}h_{\mu\mu}\geq -\bar{L}-C_{2}\|h\|_{C^{0}(\C_{\theta})}-\frac{1}{2}M.
		\end{eqnarray}
		
		Finally, \eqref{sup estimate} and \eqref{inf estimate} together yield
        	\begin{eqnarray*}
			\max\limits_{\partial \C_{\theta}}|h_{\mu\mu}|\leq \max\{L, \bar{L}\}+C_{2}\|h\|_{C^{0}(\C_{\theta})}+\frac M2.
		\end{eqnarray*} It follows
		\begin{eqnarray*}
		M=	\max\limits_{\partial \C_{\theta}}|h_{\mu\mu}|\leq 2\left(\max\{L, \bar{L}\}+C_{2}\|h\|_{C^{0}(\C_{\theta})} \right).
		\end{eqnarray*}
		If $p>1$, together with Lemma \ref{lem c2 boundary}, we obtain
		\begin{eqnarray*}
			\max\limits_{\C_{\theta}}|\n^{2}h|\leq C,
		\end{eqnarray*}
 where $C>0$ depends  on $n, p, \min\limits_{\C_{\theta}}f, \min\limits_{\C_{\theta}}h,  \|f\|_{C^{2}(\C_{\theta})}$ and $ \|h\|_{C^{0}(\C_{\theta})}$. 
	\end{proof}

Finally, we complete the proof of Theorem \ref{thm priori est}.
\begin{proof}[\textbf{Proof of Theorem~\ref{thm priori est}}]
We begin by proving parts (1) and (2) simultaneously. If $p\neq n+1$,  combining Lemma \ref{C0 pb}, Lemma \ref{C0 ps}, and Lemma  \ref{c2 est}, we obtain
    \begin{eqnarray*}
      c\leq \min\limits_{\C_{\theta}}h, \quad\text{and}\quad  \|h\|_{C^{2}(\C_{\theta})} \leq C. 
    \end{eqnarray*}
   By the theory of fully nonlinear second-order uniformly elliptic equations with oblique boundary conditions (cf. \cite[Theorem 1.1]{LT}, \cite{LTU}) and the Schauder estimate, we have \eqref{thm-C2}.
  
     Next, we prove \eqref{re-est}. If $n+1<p<n+2$, the function $\widetilde h\coloneqq \frac{h}{\min\limits_{\C_{\theta}}h}$ satisfies 
     \begin{eqnarray*}
		\begin{array}{rcll}
			\det(\n^2 \widetilde h+\widetilde h\s )&=&  (\min\limits_{\C_{\theta}}h)^{p-(n+1)} f \widetilde{h}^{p-1},& \quad    \hbox{ in } \C_{\theta},\\
			\n_\mu \widetilde h &=& \cot\theta  \widetilde h, &\quad  \hbox{ on } \p \C_\theta.
		\end{array}
\end{eqnarray*}
From  Lemma \ref{lem kC1}, we have $$1\leq \widetilde{h}\leq C ~~~~ \text{ and } ~~~~|\n \widetilde h|\leq C.$$ Lemma \ref{C0 pb} implies that $(\min\limits_{\C_{\theta}}h)^{p-n-1}$ has a uniformly positive upper and lower bounds, thus together with  Lemma \ref{c2 est}, we conclude that $$\|\wt h\|_{C^2(\C_\theta)}\leq C.$$ Therefore, using the Evans-Krylov theorem and the Schauder theory, we conclude that \eqref{re-est} holds. This completes the proof.

\end{proof}

\section{Proof of Theorem \ref{thm-1}}\label{sec-4}
\
In this section, we use the continuity method to complete the proof of Theorem \ref{thm-1}, as in \cite[Section 4]{MWW-AIM}.
Let 
\begin{eqnarray*}
    f_{t}\coloneqq (1-t)\ell^{1-p}+tf, \quad \text{for}\quad 0\leq t\leq 1.
\end{eqnarray*}
Consider the problem 
\begin{eqnarray}\label{eqt}
		\begin{array}{rcll}
			\det(\n^2 h+h\s )&=& h^{p-1} f_{t}, & \quad    \hbox{ in } \C_{\theta},\\
			\n_\mu h &=& \cot\theta  h, &\quad  \hbox{ on } \p \C_\theta.
		\end{array}
	\end{eqnarray} 
Define the set 
\begin{eqnarray*} 
    \mathcal{H}\coloneqq \left\{h\in C^{4,\alpha}(\C_{\theta}): \n_{\mu}h=\cot\theta h~\text{on}~\partial \C_{\theta} \right\},
\end{eqnarray*}
and when $p>n+1$, we denote
\begin{eqnarray*}
		\mathcal{I}\coloneqq  \left\{ t\in[0,1]: \text{Eq.}~\eqref{eqt}~\text{has~a~positive~solution~in~}\mathcal{H}  \right\},
\end{eqnarray*}
when $1<p<n+1$, we denote
\begin{eqnarray*}
 	\mathcal{I}\coloneqq  \left\{ t\in[0,1]: \text{Eq.}~\eqref{eqt}~\text{has~a~positive, capillary~even~solution~in~}\mathcal{H}  \right\}.
\end{eqnarray*}
We rewrite Eq. \eqref{eqt} as
\begin{eqnarray}\label{det-A-f-t}
    \mathcal{G}(A)\coloneqq \det(A)=h^{p-1}f_{t},
\end{eqnarray} and the linearized operator $L_{h}$ of Eq. \eqref{det-A-f-t} is 
\begin{eqnarray*}
    L_{h}(v)\coloneqq \mathcal{G}^{ij}(v_{ij}+v\sigma_{ij})-(p-1)f_{t} h^{p-2}v,
\end{eqnarray*}where $\mathcal{G}^{ij}\coloneqq\frac{\partial \mathcal{G}(A)}{\partial A_{ij}}$.
Since $h$ is a convex function, we know that $\mathcal{G}^{ij}$ is a positive definite matrix.  Next, we adopt the technique in  \cite[Section~4.3]{GX} and show that the kernel of operator $L_{h}$ is trivial when $p\neq n+1$.
\begin{lem}\label{ker lem}
Let $p\neq n+1>1$ and $\theta\in(0, \frac{\pi}{2})$. Suppose $v\in \mathcal{H}$ and $v\in {\rm{Ker}}(L_{h})$, then $v\equiv 0$.   
\end{lem}
\begin{proof}
Suppose $v\in {\rm{Ker}}(L_{h})$, then 
\begin{eqnarray}\label{ker-1}
    \mathcal{G}^{ij}(v_{ij}+v\sigma_{ij})-(p-1)h^{-1}\det(\n^{2}h+h\sigma)v=0.
\end{eqnarray}
Multiplying \eqref{ker-1} with $h$ and integrating over $\C_{\theta}$ and applying integration by parts twice,  we get
\begin{eqnarray}\label{ker-2}
    &&(p-1)\int_{\C_{\theta}}v \det(\n^{2}h+h\sigma)d\sigma =\int_{\C_{\theta}}\mathcal{G}^{ij}(v_{ij}+v\sigma_{ij})hd\sigma \notag \\ &=&n\int_{\C_{\theta}}v\det(\n^{2}h+h\sigma)d\sigma+\int_{\partial\C_{\theta}}\mathcal{G}^{ij}\left(hv_{i}\<\mu, e_{j}\>-vh_{i}\<\mu, e_{j}\>\right)ds.
\end{eqnarray}
Along $\partial\C_{\theta}$, \cite[Proposition 2.8]{MWW-AIM} implies that $\mu$ is a principal direction of the principal radii $(\n^2 h+h\s)$, namely $$(\n^2 h+h\s)(\mu, e)=0 ~~\text{ for all } e\in T\p\C_\theta.$$ Together with Robin boundary conditions of $h,v$, we derive
\begin{eqnarray}
    \int_{\partial\C_{\theta}}\mathcal{G}^{ij}\left(hv_{i}\<\mu, e_{j}\>-vh_{i}\<\mu, e_{j}\>\right)ds=\int_{\partial \C_{\theta}}\mathcal{G}^{nn}(h\n_\mu v-v\n_\mu h)ds=0.\label{ker-3}
\end{eqnarray}
Combining with \eqref{ker-2} and \eqref{ker-3}, we obtain
\begin{eqnarray*}
    (p-1)\int_{\C_{\theta}}v\det(\n^{2}h+h\sigma)d\sigma=n\int_{\C_{\theta}}v\det(\n^{2}h+h\sigma)d\sigma.
\end{eqnarray*}
Since $p\neq n+1$, we conclude that 
\begin{eqnarray}\label{ker-5}
    \int_{\C_{\theta}}v\det(\n^{2}h+h\sigma)d\sigma =0.
\end{eqnarray}
As a special form of \cite[Corollary 3.3]{MWWX}, there holds
\begin{eqnarray*}
   \left( \int_{\C_\theta} v \det(\n^2 h+h\s)d\s \right)^2\geq \int_{\C_\theta}  h\det(\n^2 h+h\s)d\s \int_{\C_\theta} v \mathcal G^{ij} (v_{ij}+v\s_{ij}) d\s, 
\end{eqnarray*} which, together with  \eqref{ker-5}, implies
\begin{eqnarray}\label{ker-6}
    \int_{\C_{\theta}}v \mathcal{G}^{ij}(v_{ij}+v\sigma_{ij})d\sigma\leq 0.
\end{eqnarray}
Similarly as above, by multiplying \eqref{ker-1} with $v$ and integrating over $\C_{\theta}$, we get
\begin{eqnarray}\label{ker-7}
    (p-1)\int_{\C_{\theta}}h^{-1}v^{2}\det(\n^{2}h+h\sigma)d\sigma=\int_{\C_{\theta}}v \mathcal{G}^{ij}(v_{ij}+v\sigma_{ij})d\sigma.
\end{eqnarray}
Due to $p>1$ and combining \eqref{ker-6} and \eqref{ker-7}, we derive
\begin{eqnarray*}
    \int_{\C_{\theta}}h^{-1}v^{2}\det(\n^{2}h+h\sigma)d\sigma=0,
\end{eqnarray*}
which implies $v\equiv 0$.
\end{proof}

We are now in a position to prove Theorem \ref{thm-1}.
\begin{proof}[\textbf{Proof of Theorem \ref{thm-1}}]
First, we show the existence parts of (1) and (3). Let  $p>1$ and $p\neq n+1$. Consider Eq. \eqref{eqt}.  For $t=0$, $h=\ell\in \mathcal{H}$ is  a solution, i.e., $0\in \mathcal I$, so the set $\mathcal{I}$ is nonempty.
The closeness of $\mathcal{I}$ follows from Theorem \ref{thm priori est}, while the openness of $\mathcal{I}$ follows from the implicit function theorem, Lemma \ref{ker lem}, and Fredholm's alternative theorem. Then it follows that $1\in \mathcal I$, namely, we obtain a positive solution to Eq. \eqref{eq-monge-ampere-Lp}.  

We move on to prove (2), namely $p=n+1$. For a fixed positive constant $\varepsilon\in (0, 1)$, consider the approximating equation:
\begin{eqnarray}\label{app-equ-1}
		\begin{array}{rcll}
			\det(\n^2 h+h\s )&=&  f h^{n+\varepsilon},& \quad    \hbox{ in } \C_{\theta},\\
			\n_\mu h &=& \cot\theta  h, &\quad  \hbox{ on } \p \C_\theta.
		\end{array}
	\end{eqnarray} 
From the preceding discussion, we know that there exists a positive function $h_{\varepsilon} \in C^{3, \alpha}(\C_{\theta})$ that solves Eq.~\eqref{app-equ-1}. Denote $$\wt{h}_{\varepsilon} \coloneqq \frac{h_{\varepsilon}}{\min\limits_{\C_{\theta}}h_{\varepsilon}},$$ from \eqref{app-equ-1},  then $\widetilde{h}_{\varepsilon}$ satisfies 
\begin{eqnarray*}
		\begin{array}{rcll}
        			\det(\n^2 h+h\s )&=&  (\min\limits_{\C_{\theta}}h_{\varepsilon})^{\varepsilon} f h^{n+\varepsilon},& \quad    \hbox{ in } \C_{\theta},\\
			\n_\mu h &=& \cot\theta  h, &\quad  \hbox{ on } \p \C_\theta.
		\end{array}
\end{eqnarray*}
Theorem \ref{thm priori est}  implies
\begin{eqnarray*}
    \|\widetilde{h}_{\varepsilon}\|_{C^{3,\alpha}}\leq C,\end{eqnarray*}
where the constant $C$ is independent of $\varepsilon$, and it only depends on $n, k, \min\limits_{\C_{\theta}}f$ and $\|f\|_{C^{2}(\C_{\theta})}$. Hence, there exists a subsequence $\varepsilon_{j}\rightarrow 0$, such that $\widetilde{h}_{\varepsilon_{j}}\rightarrow h$ in $C^{2, \alpha_{0}}(\C_{\theta})$ for any $0<\alpha_{0}<\alpha$, and Lemma \ref{C0 pb} implies that $(\min\limits_{\C_{\theta}}h_{\varepsilon_{j}})\rightarrow \gamma$  for some positive constant $\gamma$ as $j\rightarrow+\infty$. Thus we conclude that the pair $(h, \gamma)$ solve
\begin{eqnarray*}
		\begin{array}{rcll}
			\det(\n^2 h+h\s )&=&  \gamma f h^{n},& \quad    \hbox{ in } \C_{\theta},\\
			\n_\mu h &=& \cot\theta  h, &\quad  \hbox{ on } \p \C_\theta.
		\end{array}
\end{eqnarray*}
Therefore, we prove the existence part for the case $p=n+1$ of (2).

The uniqueness of the solution can be established by an argument analogous to those in \cite[Lemma 8]{GL2000}, \cite[Section 4.3]{GX}, and \cite[Section 3.2]{MWW-Quotient} for $p=n+1$, $1<p<n+1$ and $p>n+1$, respectively. We omit the details for brevity. This completes the proof.

\end{proof}

\bigskip

\subsection*{Acknowledgment} X. M. was supported by the National Key R$\&$D Program of China 2020YFA0712800  and the Postdoctoral Fellowship Program of CPSF under Grant Number 2025T180843 and 2025M773082. L. W. was partially supported by CRM De Giorgi of Scuola Normale Superiore and PRIN project 2022E9CF89 of the University of Pisa. L.W. is a member of GNAMPA as part of INdAM. We wish to express sincere appreciation to Dr. Wei Wei for sharing a proof of Lemma 3.5.

\printbibliography

@article {Alex,
    AUTHOR = {Aleksandrov, Aleksandr Danilovich},
     TITLE = {Uniqueness theorems for surfaces in the large. {I}},
   JOURNAL = {Vestnik Leningrad. Univ.},
  FJOURNAL = {Vestnik Leningrad. Univ.},
    VOLUME = {11},
      YEAR = {1956},
    NUMBER = {19},
     PAGES = {5--17},
   MRCLASS = {53.0X},
  MRNUMBER = {86338},
MRREVIEWER = {H.\ Busemann},
}

@article {ACW,
    AUTHOR = {Ai, Jun and Chou, Kai-Seng and Wei, Juncheng},
     TITLE = {Self-similar solutions for the anisotropic affine curve
              shortening problem},
   JOURNAL = {Calc. Var. Partial Differential Equations},
  FJOURNAL = {Calculus of Variations and Partial Differential Equations},
    VOLUME = {13},
      YEAR = {2001},
    NUMBER = {3},
     PAGES = {311--337},
      ISSN = {0944-2669},
   MRCLASS = {35J60 (35A30 35K65 53C44)},
  MRNUMBER = {1865001},
MRREVIEWER = {Dian K. Palagachev},
       DOI = {10.1007/s005260000075},
       URL = {https://doi.org/10.1007/s005260000075},
}

@article {Ben98,
    AUTHOR = {Andrews, Ben},
     TITLE = {Evolving convex curves},
   JOURNAL = {Calc. Var. Partial Differential Equations},
  FJOURNAL = {Calculus of Variations and Partial Differential Equations},
    VOLUME = {7},
      YEAR = {1998},
    NUMBER = {4},
     PAGES = {315--371},
      ISSN = {0944-2669},
   MRCLASS = {58E10 (35K10 53A04)},
  MRNUMBER = {1660843},
MRREVIEWER = {Dong-Ho Tsai},
       DOI = {10.1007/s005260050111},
       URL = {https://doi.org/10.1007/s005260050111},
}

@article{Böröczky_Guan_2023, 
     title={Anisotropic flow, entropy, and $L^p$-Minkowski problem}, 
     DOI={10.4153/S0008414X23000792}, 
    journal={Canadian Journal of Mathematics}, 
    author={Böröczky, Károly J and Guan, Pengfei}, 
    year={2023}, 
    pages={1–20}}

@article {BLYZ,
    AUTHOR = {B\"or\"oczky, K\'aroly J and Lutwak, Erwin and Yang, Deane
              and Zhang, Gaoyong},
     TITLE = {The logarithmic {M}inkowski problem},
   JOURNAL = {J. Amer. Math. Soc.},
  FJOURNAL = {Journal of the American Mathematical Society},
    VOLUME = {26},
      YEAR = {2013},
    NUMBER = {3},
     PAGES = {831--852},
      ISSN = {0894-0347,1088-6834},
   MRCLASS = {52A40 (52A38)},
  MRNUMBER = {3037788},
MRREVIEWER = {Alina\ Stancu},
       DOI = {10.1090/S0894-0347-2012-00741-3},
       URL = {https://doi.org/10.1090/S0894-0347-2012-00741-3},
}

@article {BHZ,
     AUTHOR = {B\"{o}r\"{o}czky, K\'{a}roly J and Heged\H{u}s, P\'{a}l and Zhu, Guangxian},
     TITLE = {On the discrete logarithmic {M}inkowski problem},
   JOURNAL = {Int. Math. Res. Not. IMRN},
  FJOURNAL = {International Mathematics Research Notices. IMRN},
      YEAR = {2016},
    NUMBER = {6},
     PAGES = {1807--1838},
      ISSN = {1073-7928},
   MRCLASS = {52A20 (52A38)},
  MRNUMBER = {3509941},
MRREVIEWER = {Mar\'{\i}a A. Hern\'{a}ndez Cifre},
       DOI = {10.1093/imrn/rnv189},
       URL = {https://doi.org/10.1093/imrn/rnv189},
}

@article {BBC,
    AUTHOR = {Bianchi, Gabriele and B\"{o}r\"{o}czky, K\'{a}roly J and Colesanti,
              Andrea},
     TITLE = {Smoothness in the {$L_p$} {M}inkowski problem for {$p<1$}},
   JOURNAL = {J. Geom. Anal.},
  FJOURNAL = {Journal of Geometric Analysis},
    VOLUME = {30},
      YEAR = {2020},
    NUMBER = {1},
     PAGES = {680--705},
      ISSN = {1050-6926},
   MRCLASS = {52A40 (35J96 52A38)},
  MRNUMBER = {4058533},
MRREVIEWER = {Simon Larson},
       DOI = {10.1007/s12220-019-00161-y},
       URL = {https://doi.org/10.1007/s12220-019-00161-y},
}

@article {BBCY,
    AUTHOR = {Bianchi, Gabriele and B\"{o}r\"{o}czky, K\'{a}roly J and Colesanti,
              Andrea and Yang, Deane},
     TITLE = {The {$L_p$}-{M}inkowski problem for {$-n<p<1$}},
   JOURNAL = {Adv. Math.},
  FJOURNAL = {Advances in Mathematics},
    VOLUME = {341},
      YEAR = {2019},
     PAGES = {493--535},
      ISSN = {0001-8708},
   MRCLASS = {52A38 (35J96)},
  MRNUMBER = {3872853},
MRREVIEWER = {Stefano Campi},
       DOI = {10.1016/j.aim.2018.10.032},
       URL = {https://doi.org/10.1016/j.aim.2018.10.032},
}

@article {Boro-Saroglou2024,
    AUTHOR = {B\"or\"oczky, K\'aroly J. and Saroglou, Christos},
     TITLE = {Uniqueness when the {$L_p$} curvature is close to be a
              constant for {$p\in [0,1)$}},
   JOURNAL = {Calc. Var. Partial Differential Equations},
  FJOURNAL = {Calculus of Variations and Partial Differential Equations},
    VOLUME = {63},
      YEAR = {2024},
    NUMBER = {6},
     PAGES = {Paper No. 154, 26},
      ISSN = {0944-2669,1432-0835},
   MRCLASS = {52A20 (52A38 52A39)},
  MRNUMBER = {4765819},
MRREVIEWER = {Alina\ Stancu},
       DOI = {10.1007/s00526-024-02763-z},
       URL = {https://doi.org/10.1007/s00526-024-02763-z},
}

@book{BRF,
    author ={B\"{o}r\"{o}czky, K\'{a}roly J and Figalli, Alessio and Ramos, João Pedro Gonçalves},
    title ={Isoperimetric Inequalities, Brunn–Minkowski Theory and Minkowski-Type Monge-Ampère Equations on the Sphere},
    SERIES = {Zurich Lectures in Advanced Mathematics},
 PUBLISHER = {EMS Press, Berlin},
      YEAR = {2026},
     PAGES = {539},
      ISBN = {978-3-98547-104-1},
       DOI = {10.4171/ZLAM/33},
       URL = {https://doi.org/10.4171/ZLAM/33},
}

@article {Bryan-Ivaki-Scheuer2019,
    AUTHOR = {Bryan, Paul and Ivaki, Mohammad N. and Scheuer, Julian},
     TITLE = {A unified flow approach to smooth, even {$L_p$}-{M}inkowski
              problems},
   JOURNAL = {Anal. PDE},
  FJOURNAL = {Analysis \& PDE},
    VOLUME = {12},
      YEAR = {2019},
    NUMBER = {2},
     PAGES = {259--280},
      ISSN = {2157-5045},
   MRCLASS = {53C44 (35K55 52A05 53A15 58J35)},
  MRNUMBER = {3861892},
MRREVIEWER = {Alina Stancu},
       DOI = {10.2140/apde.2019.12.259},
       URL = {https://doi.org/10.2140/apde.2019.12.259},
}

@article {Busemann,
    AUTHOR = {Busemann, Herbert},
     TITLE = {Minkowski's and related problems for convex surfaces with
              boundaries},
   JOURNAL = {Michigan Math. J.},
  FJOURNAL = {Michigan Mathematical Journal},
    VOLUME = {6},
      YEAR = {1959},
     PAGES = {259--266},
      ISSN = {0026-2285,1945-2365},
   MRCLASS = {53.00},
  MRNUMBER = {108829},
MRREVIEWER = {L.\ A.\ Santal\'o},
       URL = {http://projecteuclid.org/euclid.mmj/1028998232},
}

@article {CFL,
    AUTHOR = {Chen, Shibing and Feng, Yibin and Liu, Weiru},
     TITLE = {Uniqueness of solutions to the logarithmic {M}inkowski problem
              in {$\mathbb{R}^3$}},
   JOURNAL = {Adv. Math.},
  FJOURNAL = {Advances in Mathematics},
    VOLUME = {411},
      YEAR = {2022},
    NUMBER = {part A},
     PAGES = {Paper No. 108782, 18},
      ISSN = {0001-8708},
   MRCLASS = {52A40},
  MRNUMBER = {4512403},
MRREVIEWER = {Ge Xiong},
       DOI = {10.1016/j.aim.2022.108782},
       URL = {https://doi.org/10.1016/j.aim.2022.108782},
}

@article {CLZ,
    AUTHOR = {Chen, Shibing and Li, Qirui and Zhu, Guangxian},
     TITLE = {On the {$L_p$} {M}onge-{A}mp\`ere equation},
   JOURNAL = {J. Differential Equations},
  FJOURNAL = {Journal of Differential Equations},
    VOLUME = {263},
      YEAR = {2017},
    NUMBER = {8},
     PAGES = {4997--5011},
      ISSN = {0022-0396},
   MRCLASS = {35J96 (35J20 35J60 52A40)},
  MRNUMBER = {3680945},
       DOI = {10.1016/j.jde.2017.06.007},
       URL = {https://doi.org/10.1016/j.jde.2017.06.007},
}

@article {Chen-Li-Zhu2019TAMS,
    AUTHOR = {Chen, Shibing and Li, Qirui and Zhu, Guangxian},
     TITLE = {The logarithmic {M}inkowski problem for non-symmetric
              measures},
   JOURNAL = {Trans. Amer. Math. Soc.},
  FJOURNAL = {Transactions of the American Mathematical Society},
    VOLUME = {371},
      YEAR = {2019},
    NUMBER = {4},
     PAGES = {2623--2641},
      ISSN = {0002-9947},
   MRCLASS = {52A40 (35J75 35J96)},
  MRNUMBER = {3896091},
MRREVIEWER = {Andrea Colesanti},
       DOI = {10.1090/tran/7499},
       URL = {https://doi.org/10.1090/tran/7499},
}

@article {CY76,
    AUTHOR = {Cheng, Shiu Yuen and Yau, Shing Tung},
     TITLE = {On the regularity of the solution of the {$n$}-dimensional
              {M}inkowski problem},
   JOURNAL = {Comm. Pure Appl. Math.},
  FJOURNAL = {Communications on Pure and Applied Mathematics},
    VOLUME = {29},
      YEAR = {1976},
    NUMBER = {5},
     PAGES = {495--516},
      ISSN = {0010-3640,1097-0312},
   MRCLASS = {53C45 (35J60)},
  MRNUMBER = {423267},
MRREVIEWER = {H.\ W.\ Guggenheimer},
       DOI = {10.1002/cpa.3160290504},
       URL = {https://doi.org/10.1002/cpa.3160290504},
}

@article {CW2006-AIM,
    AUTHOR = {Chou, Kai-Seng and Wang, Xu-Jia},
     TITLE = {The {$L_p$}-{M}inkowski problem and the {M}inkowski problem in
              centroaffine geometry},
   JOURNAL = {Adv. Math.},
  FJOURNAL = {Advances in Mathematics},
    VOLUME = {205},
      YEAR = {2006},
    NUMBER = {1},
     PAGES = {33--83},
      ISSN = {0001-8708,1090-2082},
   MRCLASS = {52A38 (35J20 35J60 52A21 52A39 52A40 53A15)},
  MRNUMBER = {2254308},
MRREVIEWER = {Wolfgang\ Lusky},
       DOI = {10.1016/j.aim.2005.07.004},
       URL = {https://doi.org/10.1016/j.aim.2005.07.004},
}

@article {Du,
    AUTHOR = {Du, Shi-Zhong},
     TITLE = {On the planar {$L_ p$}-{M}inkowski problem},
   JOURNAL = {J. Differential Equations},
  FJOURNAL = {Journal of Differential Equations},
    VOLUME = {287},
      YEAR = {2021},
     PAGES = {37--77},
      ISSN = {0022-0396},
   MRCLASS = {35J20 (35J60 52A40 53A15)},
  MRNUMBER = {4238072},
       DOI = {10.1016/j.jde.2021.03.035},
       URL = {https://doi.org/10.1016/j.jde.2021.03.035},
}

@book {Finn,
    AUTHOR = {Finn, Robert},
     TITLE = {Equilibrium capillary surfaces},
    SERIES = {Grundlehren der mathematischen Wissenschaften [Fundamental
              Principles of Mathematical Sciences]},
    VOLUME = {284},
 PUBLISHER = {Springer-Verlag, New York},
      YEAR = {1986},
     PAGES = {xvi+245},
      ISBN = {0-387-96174-7},
   MRCLASS = {49-02 (49F10 53-02 53A10 58E12)},
  MRNUMBER = {816345},
MRREVIEWER = {Helmut\ Kaul},
       DOI = {10.1007/978-1-4613-8584-4},
       URL = {https://doi.org/10.1007/978-1-4613-8584-4},
}

@article {Firey,
    AUTHOR = {Firey, Wm J},
     TITLE = {{$p$}-means of convex bodies},
   JOURNAL = {Math. Scand.},
  FJOURNAL = {Mathematica Scandinavica},
    VOLUME = {10},
      YEAR = {1962},
     PAGES = {17--24},
      ISSN = {0025-5521},
   MRCLASS = {52.30},
  MRNUMBER = {141003},
MRREVIEWER = {D. Derry},
       DOI = {10.7146/math.scand.a-10510},
       URL = {https://doi.org/10.7146/math.scand.a-10510},
}

@article {Gb99,
    AUTHOR = {Guan, Bo},
     TITLE = {The {D}irichlet problem for {H}essian equations on Riemannian manifolds},
   JOURNAL = {Calc. Var. Partial Differential Equations},
  FJOURNAL = {Calculus of Variations and Partial Differential Equations},
    VOLUME = {8},
      YEAR = {1999},
    NUMBER = {1},
     PAGES = {45--69},
      ISSN = {0944-2669,1432-0835},
   MRCLASS = {58G30 (35J65 53C21)},
  MRNUMBER = {1666866},
MRREVIEWER = {John\ Urbas},
       DOI = {10.1007/s005260050116},
       URL = {https://doi.org/10.1007/s005260050116},
}

@article {GLW,
    AUTHOR = {Guang, Qiang and Li, Qi-Rui and Wang, Xu-Jia},
     TITLE = {Existence of convex hypersurfaces with prescribed centroaffine
              curvature},
   JOURNAL = {Trans. Amer. Math. Soc.},
  FJOURNAL = {Transactions of the American Mathematical Society},
    VOLUME = {377},
      YEAR = {2024},
    NUMBER = {2},
     PAGES = {841--862},
      ISSN = {0002-9947},
   MRCLASS = {35J20 (35K96 53A07)},
  MRNUMBER = {4688536},
       DOI = {10.1090/tran/8881},
       URL = {https://doi.org/10.1090/tran/8881},
}

@article{GLW-1,    
AUTHOR = {Guang, Qiang and Li, Qi-Rui and Wang, Xu-Jia},
     TITLE = {The {$L_p$}-{M}inkowski problem with super-critical exponents},
   JOURNAL = {J. Eur. Math. Soc. (JEMS)},
  FJOURNAL = {Journal of the European Mathematical Society (JEMS)},
    VOLUME = {28},
      YEAR = {2026},
    NUMBER = {2},
     PAGES = {735--775},
      ISSN = {1435-9855,1435-9863},
   MRCLASS = {35J20 (35K96 52A20 53A07)},
  MRNUMBER = {5031428},
       DOI = {10.4171/jems/1733},
       URL = {https://doi.org/10.4171/jems/1733},
}

@article {HLW,
    AUTHOR = {He, Yan and Li, Qi-Rui and Wang, Xu-Jia},
     TITLE = {Multiple solutions of the {$L_p$}-{M}inkowski problem},
   JOURNAL = {Calc. Var. Partial Differential Equations},
  FJOURNAL = {Calculus of Variations and Partial Differential Equations},
    VOLUME = {55},
      YEAR = {2016},
    NUMBER = {5},
     PAGES = {Art. 117, 13},
      ISSN = {0944-2669},
   MRCLASS = {35J96 (52A39)},
  MRNUMBER = {3551297},
MRREVIEWER = {Xiaobing Henry Feng},
       DOI = {10.1007/s00526-016-1063-y},
       URL = {https://doi.org/10.1007/s00526-016-1063-y},
}

@article {HMS2004,
    AUTHOR = {Hu, Changqing and Ma, Xi-Nan and Shen, Chunli},
     TITLE = {On the {C}hristoffel-{M}inkowski problem of {F}irey's
              {$p$}-sum},
   JOURNAL = {Calc. Var. Partial Differential Equations},
  FJOURNAL = {Calculus of Variations and Partial Differential Equations},
    VOLUME = {21},
      YEAR = {2004},
    NUMBER = {2},
     PAGES = {137--155},
      ISSN = {0944-2669},
   MRCLASS = {52A39},
  MRNUMBER = {2085300},
MRREVIEWER = {Andrea Colesanti},
       DOI = {10.1007/s00526-003-0250-9},
       URL = {https://doi.org/10.1007/s00526-003-0250-9},
}

@article {Huang-Lu2013,
    AUTHOR = {Huang, Yong and Lu, QiuPing},
     TITLE = {On the regularity of the {$L_p$} {M}inkowski problem},
   JOURNAL = {Adv. in Appl. Math.},
  FJOURNAL = {Advances in Applied Mathematics},
    VOLUME = {50},
      YEAR = {2013},
    NUMBER = {2},
     PAGES = {268--280},
      ISSN = {0196-8858},
   MRCLASS = {35J60 (35B65 52A40)},
  MRNUMBER = {3003347},
MRREVIEWER = {Barbara Brandolini},
       DOI = {10.1016/j.aam.2012.08.005},
       URL = {https://doi.org/10.1016/j.aam.2012.08.005},
}

@article {Hug,
    AUTHOR = {Hug, Daniel},
     TITLE = {Contributions to affine surface area},
   JOURNAL = {Manuscripta Math.},
  FJOURNAL = {Manuscripta Mathematica},
    VOLUME = {91},
      YEAR = {1996},
    NUMBER = {3},
     PAGES = {283--301},
      ISSN = {0025-2611},
   MRCLASS = {52A38 (52A20)},
  MRNUMBER = {1416712},
MRREVIEWER = {Carla Peri},
       DOI = {10.1007/BF02567955},
       URL = {https://doi.org/10.1007/BF02567955},
}

@article {JLZ,
    AUTHOR = {Jian, Huaiyu and Lu, Jian and Zhu, Guangxian},
     TITLE = {Mirror symmetric solutions to the centro-affine {M}inkowski
              problem},
   JOURNAL = {Calc. Var. Partial Differential Equations},
  FJOURNAL = {Calculus of Variations and Partial Differential Equations},
    VOLUME = {55},
      YEAR = {2016},
    NUMBER = {2},
     PAGES = {Art. 41, 22},
      ISSN = {0944-2669},
   MRCLASS = {35J96 (34C40 35A01 35B06 35B44 35J75 53A15)},
  MRNUMBER = {3479715},
MRREVIEWER = {Jingang Xiong},
       DOI = {10.1007/s00526-016-0976-9},
       URL = {https://doi.org/10.1007/s00526-016-0976-9},
}

@article {Kolesnikov-milman2022,
    AUTHOR = {Kolesnikov, Alexander V. and Milman, Emanuel},
     TITLE = {Local {$L^p$}-{B}runn-{M}inkowski inequalities for {$p<1$}},
   JOURNAL = {Mem. Amer. Math. Soc.},
  FJOURNAL = {Memoirs of the American Mathematical Society},
    VOLUME = {277},
      YEAR = {2022},
    NUMBER = {1360},
     PAGES = {v+78},
      ISSN = {0065-9266,1947-6221},
   MRCLASS = {52A40 (35P15 52A23 58J50)},
  MRNUMBER = {4438690},
       DOI = {10.1090/memo/1360},
       URL = {https://doi.org/10.1090/memo/1360},
}

@article {LH,
    AUTHOR = {Lu, Jian and Jian, Huaiyu},
     TITLE = {Topological degree method for the rotationally symmetric
              {$L_p$}-{M}inkowski problem},
   JOURNAL = {Discrete Contin. Dyn. Syst.},
  FJOURNAL = {Discrete and Continuous Dynamical Systems. Series A},
    VOLUME = {36},
      YEAR = {2016},
    NUMBER = {2},
     PAGES = {971--980},
      ISSN = {1078-0947},
   MRCLASS = {35J96 (35A16 35J75 53A15)},
  MRNUMBER = {3392914},
MRREVIEWER = {Anna Maria Candela},
       DOI = {10.3934/dcds.2016.36.971},
       URL = {https://doi.org/10.3934/dcds.2016.36.971},
}

@article {LW,
    AUTHOR = {Lu, Jian and Wang, Xu-Jia},
     TITLE = {Rotationally symmetric solutions to the {$L_p$}-{M}inkowski
              problem},
   JOURNAL = {J. Differential Equations},
  FJOURNAL = {Journal of Differential Equations},
    VOLUME = {254},
      YEAR = {2013},
    NUMBER = {3},
     PAGES = {983--1005},
      ISSN = {0022-0396},
   MRCLASS = {35J96 (35B06 35B45 53A15 53A40)},
  MRNUMBER = {2997361},
MRREVIEWER = {Yuxin Ge},
       DOI = {10.1016/j.jde.2012.10.008},
       URL = {https://doi.org/10.1016/j.jde.2012.10.008},
}

@article {Lu,
    AUTHOR = {Lu, Jian},
     TITLE = {Nonexistence of maximizers for the functional of the
              centroaffine {M}inkowski problem},
   JOURNAL = {Sci. China Math},
  FJOURNAL = {Science China. Mathematics},
    VOLUME = {61},
      YEAR = {2018},
    NUMBER = {3},
     PAGES = {511--516},
      ISSN = {1674-7283},
   MRCLASS = {35J96 (34C40 35A01 35J75 53A15)},
  MRNUMBER = {3762239},
MRREVIEWER = {Xiaobing Henry Feng},
       DOI = {10.1007/s11425-016-0539-x},
       URL = {https://doi.org/10.1007/s11425-016-0539-x},
}

@article {Ludwig, 
    AUTHOR = {Ludwig, Monika},
     TITLE = {General affine surface areas},
   JOURNAL = {Adv. Math.},
  FJOURNAL = {Advances in Mathematics},
    VOLUME = {224},
      YEAR = {2010},
    NUMBER = {6},
     PAGES = {2346--2360},
      ISSN = {0001-8708},
   MRCLASS = {52A20},
  MRNUMBER = {2652209},
       DOI = {10.1016/j.aim.2010.02.004},
       URL = {https://doi.org/10.1016/j.aim.2010.02.004},
}

@article {GL2000,
    AUTHOR = {Guan, Pengfei and Lin, Changshou},
     TITLE = {On the equation $\det(u_{ij}+u\delta_{ij})=u^{p}f$ on $\mathbb{S}^{n}$},
   JOURNAL = {Preprint}, 
      YEAR = {1999},
}

@article {GX,
    AUTHOR = {Guan, Pengfei and Xia, Chao},
     TITLE = {{$L^p$} {C}hristoffel-{M}inkowski problem: the case
              {$1<p<k+1$}},
   JOURNAL = {Calc. Var. Partial Differential Equations},
  FJOURNAL = {Calculus of Variations and Partial Differential Equations},
    VOLUME = {57},
      YEAR = {2018},
    NUMBER = {2},
     PAGES = {Paper No. 69, 23},
      ISSN = {0944-2669},
   MRCLASS = {58J05 (52A20 52A39)},
  MRNUMBER = {3776359},
MRREVIEWER = {Andrea Colesanti},
       DOI = {10.1007/s00526-018-1341-y},
       URL = {https://doi.org/10.1007/s00526-018-1341-y},
}

@article {HWYZ,
    AUTHOR = {Hu, Yingxiang and Wei, Yong and Yang, Bo and Zhou, Tailong},
     TITLE = {A complete family of {A}lexandrov-{F}enchel inequalities for
              convex capillary hypersurfaces in the half-space},
   JOURNAL = {Math. Ann.},
  FJOURNAL = {Mathematische Annalen},
    VOLUME = {390},
      YEAR = {2024},
    NUMBER = {2},
     PAGES = {3039--3075},
      ISSN = {0025-5831},
   MRCLASS = {53E10 (35K93 35R01 52A40 53C21)},
  MRNUMBER = {4801847},
MRREVIEWER = {Jinyu Guo},
       DOI = {10.1007/s00208-024-02841-9},
       URL = {https://doi.org/10.1007/s00208-024-02841-9},
}

@article {HLYZ,
    AUTHOR = {Huang, Yong and Lutwak, Erwin and Yang, Deane and Zhang,
              Gaoyong},
     TITLE = {Geometric measures in the dual {B}runn-{M}inkowski theory and
              their associated {M}inkowski problems},
   JOURNAL = {Acta Math.},
  FJOURNAL = {Acta Mathematica},
    VOLUME = {216},
      YEAR = {2016},
    NUMBER = {2},
     PAGES = {325--388},
      ISSN = {0001-5962,1871-2509},
   MRCLASS = {52A38 (35J20 35J96)},
  MRNUMBER = {3573332},
MRREVIEWER = {Ai-jun\ Li},
       DOI = {10.1007/s11511-016-0140-6},
       URL = {https://doi.org/10.1007/s11511-016-0140-6},
}

@article {Hug-LYZ,
    AUTHOR = {Hug, Daniel and Lutwak, Erwin and Yang, Deane and Zhang,
              Gaoyong},
     TITLE = {On the {$L_p$} {M}inkowski problem for polytopes},
   JOURNAL = {Discrete Comput. Geom.},
  FJOURNAL = {Discrete \& Computational Geometry. An International Journal
              of Mathematics and Computer Science},
    VOLUME = {33},
      YEAR = {2005},
    NUMBER = {4},
     PAGES = {699--715},
      ISSN = {0179-5376,1432-0444},
   MRCLASS = {52A39},
  MRNUMBER = {2132298},
MRREVIEWER = {Mar\'ia\ A.\ Hern\'andez Cifre},
       DOI = {10.1007/s00454-004-1149-8},
       URL = {https://doi.org/10.1007/s00454-004-1149-8},
}

@article {Jian-Lu-Wang2015-AIM,
    AUTHOR = {Jian, Huaiyu and Lu, Jian and Wang, Xu-Jia},
     TITLE = {Nonuniqueness of solutions to the {$L_p$}-{M}inkowski problem},
   JOURNAL = {Adv. Math.},
  FJOURNAL = {Advances in Mathematics},
    VOLUME = {281},
      YEAR = {2015},
     PAGES = {845--856},
      ISSN = {0001-8708},
   MRCLASS = {35J60 (35A02 35J96 52A20)},
  MRNUMBER = {3366854},
MRREVIEWER = {Yuxin Ge},
       DOI = {10.1016/j.aim.2015.05.010},
       URL = {https://doi.org/10.1016/j.aim.2015.05.010},
}

@article {Lewy,
    AUTHOR = {Lewy, Hans},
     TITLE = {On differential geometry in the large. {I}. {M}inkowski's
              problem},
   JOURNAL = {Trans. Amer. Math. Soc.},
  FJOURNAL = {Transactions of the American Mathematical Society},
    VOLUME = {43},
      YEAR = {1938},
    NUMBER = {2},
     PAGES = {258--270},
      ISSN = {0002-9947,1088-6850},
   MRCLASS = {52A15 (35J15 52A39)},
  MRNUMBER = {1501942},
       DOI = {10.2307/1990042},
       URL = {https://doi.org/10.2307/1990042},
}

@article {LT,
    AUTHOR = {Lieberman, Gary M. and Trudinger, Neil S},
     TITLE = {Nonlinear oblique boundary value problems for nonlinear
              elliptic equations},
   JOURNAL = {Trans. Amer. Math. Soc.},
  FJOURNAL = {Transactions of the American Mathematical Society},
    VOLUME = {295},
      YEAR = {1986},
    NUMBER = {2},
     PAGES = {509--546},
      ISSN = {0002-9947,1088-6850},
   MRCLASS = {35J65},
  MRNUMBER = {833695},
MRREVIEWER = {Pierre-Louis\ Lions},
       DOI = {10.2307/2000050},
       URL = {https://doi.org/10.2307/2000050},
}

@article {LTU,
    AUTHOR = {Lions, Pierre-Louis and Trudinger, Neil Sidney and Urbas, John Ivan Eugene},
     TITLE = {The {N}eumann problem for equations of {M}onge-{A}mp\`ere
              type},
   JOURNAL = {Comm. Pure Appl. Math.},
  FJOURNAL = {Communications on Pure and Applied Mathematics},
    VOLUME = {39},
      YEAR = {1986},
    NUMBER = {4},
     PAGES = {539--563},
      ISSN = {0010-3640,1097-0312},
   MRCLASS = {35J25 (35B45 35B50 53C45)},
  MRNUMBER = {840340},
MRREVIEWER = {Philippe\ Delano\"e},
       DOI = {10.1002/cpa.3160390405},
       URL = {https://doi.org/10.1002/cpa.3160390405},
}

@article {Lut93,
    AUTHOR = {Lutwak, Erwin},
     TITLE = {The {B}runn-{M}inkowski-{F}irey theory. {I}. {M}ixed volumes
              and the {M}inkowski problem},
   JOURNAL = {J. Differential Geom.},
  FJOURNAL = {Journal of Differential Geometry},
    VOLUME = {38},
      YEAR = {1993},
    NUMBER = {1},
     PAGES = {131--150},
      ISSN = {0022-040X,1945-743X},
   MRCLASS = {52A39},
  MRNUMBER = {1231704},
MRREVIEWER = {P.\ R.\ Goodey},
       URL = {http://projecteuclid.org/euclid.jdg/1214454097},
}

@article {LO95,
    AUTHOR = {Lutwak, Erwin and Oliker, Vladimir},
     TITLE = {On the regularity of solutions to a generalization of the
              {M}inkowski problem},
   JOURNAL = {J. Differential Geom.},
  FJOURNAL = {Journal of Differential Geometry},
    VOLUME = {41},
      YEAR = {1995},
    NUMBER = {1},
     PAGES = {227--246},
      ISSN = {0022-040X,1945-743X},
   MRCLASS = {52A20 (52A39 53A07)},
  MRNUMBER = {1316557},
MRREVIEWER = {Carla\ Peri},
       URL = {http://projecteuclid.org/euclid.jdg/1214456011},
}

@article {Lutwak-Yang-Zhang2004,
    AUTHOR = {Lutwak, Erwin and Yang, Deane and Zhang, Gaoyong},
     TITLE = {On the {$L_p$}-{M}inkowski problem},
   JOURNAL = {Trans. Amer. Math. Soc.},
  FJOURNAL = {Transactions of the American Mathematical Society},
    VOLUME = {356},
      YEAR = {2004},
    NUMBER = {11},
     PAGES = {4359--4370},
      ISSN = {0002-9947},
   MRCLASS = {52A40},
  MRNUMBER = {2067123},
MRREVIEWER = {Friedrich Pillichshammer},
       DOI = {10.1090/S0002-9947-03-03403-2},
       URL = {https://doi.org/10.1090/S0002-9947-03-03403-2},
}

@article {MX,
    AUTHOR = {Ma, Xi-Nan and Xu, Jinju},
     TITLE = {Gradient estimates of mean curvature equations with {N}eumann
              boundary value problems},
   JOURNAL = {Adv. Math.},
  FJOURNAL = {Advances in Mathematics},
    VOLUME = {290},
      YEAR = {2016},
     PAGES = {1010--1039},
      ISSN = {0001-8708},
   MRCLASS = {35J93 (35B45 35B50)},
  MRNUMBER = {3451945},
MRREVIEWER = {David James Hartley},
       DOI = {10.1016/j.aim.2015.10.031},
       URL = {https://doi.org/10.1016/j.aim.2015.10.031},
}

@article {MQ,
    AUTHOR = {Ma, Xi-Nan and Qiu, Guohuan},
     TITLE = {The {N}eumann problem for {H}essian equations},
   JOURNAL = {Comm. Math. Phys.},
  FJOURNAL = {Communications in Mathematical Physics},
    VOLUME = {366},
      YEAR = {2019},
    NUMBER = {1},
     PAGES = {1--28},
      ISSN = {0010-3616,1432-0916},
   MRCLASS = {35J60 (35A09)},
  MRNUMBER = {3919441},
MRREVIEWER = {Gabrielle\ Nornberg},
       DOI = {10.1007/s00220-019-03339-1},
       URL = {https://doi.org/10.1007/s00220-019-03339-1},
}

@book {Maggi,
    AUTHOR = {Maggi, Francesco},
     TITLE = {Sets of finite perimeter and geometric variational problems},
    SERIES = {Cambridge Studies in Advanced Mathematics},
    VOLUME = {135},
      NOTE = {An introduction to geometric measure theory},
 PUBLISHER = {Cambridge University Press, Cambridge},
      YEAR = {2012},
     PAGES = {xx+454},
      ISBN = {978-1-107-02103-7},
   MRCLASS = {49-01 (26B20 28-02 49-02 49Q05 49Q20)},
  MRNUMBER = {2976521},
MRREVIEWER = {Giovanni\ Alberti},
       DOI = {10.1017/CBO9781139108133},
       URL = {https://doi.org/10.1017/CBO9781139108133},
}

@article {MWW24-IMRN,
    AUTHOR = {Mei, Xinqun and Wang, Guofang and Weng, Liangjun},
     TITLE = {A constrained mean curvature flow and Alexandrov-Fenchel inequalities},
   JOURNAL = {Int. Math. Res. Not. IMRN},
  FJOURNAL = {International Mathematics Research Notices. IMRN},
      YEAR = {2024},
    NUMBER = {1},
     PAGES = {152--174},
      ISSN = {1073-7928,1687-0247},
   MRCLASS = {53E10 (52A40)},
  MRNUMBER = {4686648},
       DOI = {10.1093/imrn/rnad020},
       URL = {https://doi.org/10.1093/imrn/rnad020},
}

@article {MWW-AIM,
    AUTHOR = {Mei, Xinqun and Wang, Guofang and Weng, Liangjun},
     TITLE = {The capillary {M}inkowski problem},
   JOURNAL = {Adv. Math.},
  FJOURNAL = {Advances in Mathematics},
    VOLUME = {469},
      YEAR = {2025},
     PAGES = {Paper No. 110230},
      ISSN = {0001-8708},
   MRCLASS = {53C65 (35J60 35J66 53C42 53C45)},
  MRNUMBER = {4884106},
       DOI = {10.1016/j.aim.2025.110230},
       URL = {https://doi.org/10.1016/j.aim.2025.110230},
}

@article{MWW-Lp-CM,
      title={Prescribed $L_p$ curvature problem for convex capillary hypersurface}, 
      author={Xinqun Mei and Guofang Wang and Liangjun Weng},
      year={2025},
      eprint={2512.16686},
      archivePrefix={arXiv},
      primaryClass={math.DG},
      url={https://arxiv.org/abs/2512.16686}, 
}

@article{MWW-Quotient,
      author={Mei, Xinqun and Wang, Guofang and Weng, Liangjun},
      title={Prescribed $L_p$ quotient curvature problem and related eigenvalue problem}, 
      year={2024},
      eprint={2402.12314},
      archivePrefix={arXiv},
      primaryClass={math.AP},
      url={https://arxiv.org/abs/2402.12314}, 
}

@article {MWWX,
 title={Alexandrov-Fenchel inequalities for convex hypersurfaces in the half-space with capillary boundary II}, 
      author={Xinqun Mei and Guofang Wang and Liangjun Weng and Chao Xia},
   JOURNAL = {Math. Z.},
  FJOURNAL = {Mathematische Zeitschrift},
    VOLUME = {310},
      YEAR = {2025},
    NUMBER = {4},
     PAGES = {Paper No. 71},
      ISSN = {0025-5874,1432-1823},
       DOI = {10.1007/s00209-025-03781-z},
       URL = {https://doi.org/10.1007/s00209-025-03781-z},
}

@article {MWW-CM,
  AUTHOR = {Mei, Xinqun and Wang, Guofang and Weng, Liangjun},
     TITLE = {The capillary Christoffel–Minkowski problem},
JOURNAL = {Calc. Var. Partial Differential Equations},
  FJOURNAL = {Calculus of Variations and Partial Differential Equations},
    VOLUME = {65},
      YEAR = {2026},
    NUMBER = {6},
     PAGES = {Paper No. 186},
      ISSN = {0944-2669,1432-0835},
       DOI = {10.1007/s00526-026-03355-9},
       URL = {https://doi.org/10.1007/s00526-026-03355-9},
}

@article{MWW-GCF,
title={The capillary Gauss curvature flow}, 
      author={Xinqun Mei and Guofang Wang and Liangjun Weng},
      year={2025},
      eprint={2506.09840},
      archivePrefix={arXiv},
      primaryClass={math.DG},
      url={https://arxiv.org/abs/2506.09840}, 
}

@article{mei-wang-weng-2025convexcapillaryhypersurfacesprescribed,
      title={Convex capillary hypersurfaces of prescribed curvature problem}, 
      author={Xinqun Mei and Guofang Wang and Liangjun Weng},
      year={2025},
      eprint={2504.14392},
      archivePrefix={arXiv},
      primaryClass={math.DG},
      url={https://arxiv.org/abs/2504.14392}, 
}

@article {Min,
    AUTHOR = {Minkowski, Hermann},
     TITLE = {Allgemeine Lehrsätze über die konvexen Polyeder},
   JOURNAL = {Nachr. Ges. Wiss. Gottingen},
    YEAR = {1897},
     PAGES = {198–219},
}

@article {Nire,
    AUTHOR = {Nirenberg, Louis},
     TITLE = {The {W}eyl and {M}inkowski problems in differential geometry
              in the large},
   JOURNAL = {Comm. Pure Appl. Math.},
  FJOURNAL = {Communications on Pure and Applied Mathematics},
    VOLUME = {6},
      YEAR = {1953},
     PAGES = {337--394},
      ISSN = {0010-3640,1097-0312},
   MRCLASS = {53.0X},
  MRNUMBER = {58265},
MRREVIEWER = {H.\ Busemann},
       DOI = {10.1002/cpa.3160060303},
       URL = {https://doi.org/10.1002/cpa.3160060303},
}

@article {Pog52,
    AUTHOR = {Pogorelov, Aleksey Vasil\cprime},
     TITLE = {Regularity of a convex surface with given {G}aussian
              curvature},
   JOURNAL = {Mat. Sbornik N.S.},
  FJOURNAL = {Mat. Sbornik N.S.},
    VOLUME = {31/73},
      YEAR = {1952},
     PAGES = {88--103},
   MRCLASS = {52.0X},
  MRNUMBER = {52807},
MRREVIEWER = {H.\ Busemann},
}

@book {Sch,
    AUTHOR = {Schneider, Rolf},
     TITLE = {Convex bodies: the {B}runn-{M}inkowski theory},
    SERIES = {Encyclopedia of Mathematics and its Applications},
    VOLUME = {151},
   EDITION = {expanded},
 PUBLISHER = {Cambridge University Press, Cambridge},
      YEAR = {2014},
     PAGES = {xxii+736},
      ISBN = {978-1-107-60101-7},
 URL={https://doi.org/10.1017/CBO9781139003858},
   MRCLASS = {52-02 (52A20 52A39)},
  MRNUMBER = {3155183},
MRREVIEWER = {Andrea Colesanti},
}

@article {Sta-1,
    AUTHOR = {Stancu, Alina},
     TITLE = {On the number of solutions to the discrete two-dimensional
              {$L_0$}-{M}inkowski problem},
   JOURNAL = {Adv. Math.},
  FJOURNAL = {Advances in Mathematics},
    VOLUME = {180},
      YEAR = {2003},
    NUMBER = {1},
     PAGES = {290--323},
      ISSN = {0001-8708},
   MRCLASS = {52A10 (53C44)},
  MRNUMBER = {2019226},
MRREVIEWER = {Lyuba S. Alboul},
       DOI = {10.1016/S0001-8708(03)00005-7},
       URL = {https://doi.org/10.1016/S0001-8708(03)00005-7},
}

@article {Sta-2,
    AUTHOR = {Stancu, Alina},
     TITLE = {The necessary condition for the discrete {$L_0$}-{M}inkowski
              problem in {$\mathbb{R}^2$}},
   JOURNAL = {J. Geom.},
  FJOURNAL = {Journal of Geometry},
    VOLUME = {88},
      YEAR = {2008},
    NUMBER = {1-2},
     PAGES = {162--168},
      ISSN = {0047-2468},
   MRCLASS = {52A10 (28A75)},
  MRNUMBER = {2398486},
       DOI = {10.1007/s00022-007-1937-4},
       URL = {https://doi.org/10.1007/s00022-007-1937-4},
}

@article {SW,
    AUTHOR = {Sinestrari, Carlo and Weng, Liangjun},
     TITLE = {Hypersurfaces with capillary boundary evolving by volume
              preserving power mean curvature flow},
   JOURNAL = {Calc. Var. Partial Differential Equations},
  FJOURNAL = {Calculus of Variations and Partial Differential Equations},
    VOLUME = {63},
      YEAR = {2024},
    NUMBER = {9},
     PAGES = {Paper No. 237, 27},
      ISSN = {0944-2669},
   MRCLASS = {53E10 (35B40 35K93)},
  MRNUMBER = {4821886},
       DOI = {10.1007/s00526-024-02838-x},
       URL = {https://doi.org/10.1007/s00526-024-02838-x},
}

@article {WWX-MA2024,
    AUTHOR = {Wang, Guofang and Weng, Liangjun and Xia, Chao},
     TITLE = {Alexandrov-{F}enchel inequalities for convex hypersurfaces in
              the half-space with capillary boundary},
   JOURNAL = {Math. Ann.},
  FJOURNAL = {Mathematische Annalen},
    VOLUME = {388},
      YEAR = {2024},
    NUMBER = {2},
     PAGES = {2121--2154},
      ISSN = {0025-5831},
   MRCLASS = {53E40 (35K96 53C21 53C24)},
  MRNUMBER = {4700391},
       DOI = {10.1007/s00208-023-02571-4},
       URL = {https://doi.org/10.1007/s00208-023-02571-4},
}

@article{Weng-Xia2022,
    AUTHOR = {Weng, Liangjun and Xia, Chao},
     TITLE = {Alexandrov-{F}enchel inequality for convex hypersurfaces with
              capillary boundary in a ball},
   JOURNAL = {Trans. Amer. Math. Soc.},
  FJOURNAL = {Transactions of the American Mathematical Society},
    VOLUME = {375},
      YEAR = {2022},
    NUMBER = {12},
     PAGES = {8851--8883},
      ISSN = {0002-9947},
   MRCLASS = {53C21 (35K96 52A40)},
  MRNUMBER = {4504655},
MRREVIEWER = {Changwei Xiong},
       DOI = {10.1090/tran/8756},
       URL = {https://doi.org/10.1090/tran/8756},
}

@article {Zhu,
    AUTHOR = {Zhu, Guangxian},
     TITLE = {The logarithmic {M}inkowski problem for polytopes},
   JOURNAL = {Adv. Math.},
  FJOURNAL = {Advances in Mathematics},
    VOLUME = {262},
      YEAR = {2014},
     PAGES = {909--931},
      ISSN = {0001-8708},
   MRCLASS = {52A40},
  MRNUMBER = {3228445},
MRREVIEWER = {Thomas Wannerer},
       DOI = {10.1016/j.aim.2014.06.004},
       URL = {https://doi.org/10.1016/j.aim.2014.06.004},
}

@article {Zhu-2,
    AUTHOR = {Zhu, Guangxian},
     TITLE = {The {$L_p$} {M}inkowski problem for polytopes for {$p<0$}},
   JOURNAL = {Indiana Univ. Math. J.},
  FJOURNAL = {Indiana University Mathematics Journal},
    VOLUME = {66},
      YEAR = {2017},
    NUMBER = {4},
     PAGES = {1333--1350},
      ISSN = {0022-2518},
   MRCLASS = {52A20 (52A38 52A40)},
  MRNUMBER = {3689334},
MRREVIEWER = {David Alonso},
       DOI = {10.1512/iumj.2017.66.6110},
       URL = {https://doi.org/10.1512/iumj.2017.66.6110},
}

\end{document}